\documentclass{amsart}
\usepackage{amsthm,xcolor}
\usepackage{latexsym,amssymb,amsfonts,amsmath ,graphicx, url, paralist}
\usepackage[vcentermath,enableskew]{youngtab}
\usepackage{color}
\usepackage[noadjust]{cite} % the spacing of package cite solves a problem we don't have and introduces a problem we do.
\usepackage{hyperref, mathrsfs, tikz}
\usetikzlibrary{arrows,decorations.pathreplacing}
\usepackage[margin=1in]{geometry}
\usepackage{bm}
\usepackage{subfig}
\usepackage{mathtools}

\def\aff{{\rm aff}}
\def\om{{\rm out}}
\def\i{{\rm in }}

\newtheorem{theorem}{Theorem}[section]
\newtheorem{proposition}[theorem]{Proposition}
\newtheorem{lemma}[theorem]{Lemma}

\newtheorem{corollary}[theorem]{Corollary}
\theoremstyle{definition}
\newtheorem{definition}[theorem]{Definition}

\theoremstyle{remark}
\newtheorem{remark}[theorem]{Remark}

\newcommand{\F}{\mathcal{F}}

\DeclareMathOperator{\GT}{GT}

\def\out{{\rm outd}}

\def\ind{{\rm ind}}

\newtheoremstyle{named}{}{}{\itshape}{}{\bfseries}{.}{.5em}{#1 \thmnote{#3}}
\theoremstyle{named}
\newtheorem*{namedtheorem}{Theorem}

\newcommand\multiset[2]%
{\mathchoice{\left(\kern-0.5em{\binom{#1}{#2}}\kern-0.5em\right)}
            {\bigl(\kern-0.3em{\binom{#1}{#2}}\kern-0.3em\bigr)}
            {\bigl(\kern-0.3em{\binom{#1}{#2}}\kern-0.3em\bigr)}
            {\bigl(\kern-0.3em{\binom{#1}{#2}}\kern-0.3em\bigr)}}

\title[]{Gelfand-Tsetlin polytopes: a story of flow \& order polytopes}
\author{Ricky Ini Liu}
\address{Ricky Ini Liu, Department of Mathematics, North Carolina State University, Raleigh, NC 27695. \newline{riliu@ncsu.edu}}

\author{Karola M\'esz\'aros}
\address{Karola M\'esz\'aros, Department of Mathematics, Cornell University, Ithaca, NY 14853 and School of Mathematics, Institute for Advanced Study, Princeton, NJ 08540.  \newline{karola@math.cornell.edu}
}

\author{Avery St. Dizier}
\address{Avery St. Dizier, Department of Mathematics, Cornell University, Ithaca NY 14853.  \newline{ajs624@cornell.edu}
}

\thanks{Liu  is partially supported by a National Science Foundation Grant (DMS 1758187). 
 M\'esz\'aros is partially supported by a National Science Foundation Grant (DMS 1501059)   as well as by a von Neumann Fellowship at the IAS   funded by the Fund for Mathematics and the Friends of the Institute for Advanced Study.}

\begin{document}

\begin{abstract}
	 Gelfand-Tsetlin polytopes are prominent objects in algebraic combinatorics. The number of integer points of the Gelfand-Tsetlin polytope $\GT(\lambda)$ is equal to the dimension of the corresponding irreducible representation of $GL(n)$. It is well-known that the Gelfand-Tsetlin polytope is a marked order polytope; the authors have recently shown it to be a  flow polytope. In this paper, we draw corollaries from this result and establish a general theory connecting marked order polytopes and flow polytopes.
\end{abstract}
%\date{\today}
\maketitle

\section{Introduction}
\label{sec:intro}

Given a partition $\lambda = (\lambda_1,\dots,\lambda_n)\in \mathbb{Z}^n_{\geq 0}$, the \textbf{Gelfand-Tsetlin polytope} $\GT(\lambda)$ is the set of all nonnegative triangular arrays
	\begin{center}
		\begin{tabular}{ccccccc}
			$x_{11}$&&$x_{12}$&&$\cdots$&&$x_{1n}$\\
			&$x_{22}$&&$x_{23}$&$\cdots$&$x_{2n}$&\\
			&&$\cdots$&&$\cdots$&&\\
			&&$x_{n-1,n-1}$&&$x_{n-1,n}$&&\\
			&&&$x_{nn}$&&&
		\end{tabular}
	\end{center}
	such that 
	\begin{align*}
		x_{in}=\lambda_i &\mbox{ for all } 1\leq i\leq n,\\
		x_{i-1,j-1}\geq x_{ij}\geq x_{i-1,j} &\mbox{ for all } 1\leq i \leq j\leq n.
	\end{align*} 

The integer points of $\GT(\lambda)$ are in bijection with semistandard Young tableaux of shape $\lambda$ on the alphabet $[n]$.   Moreover, the integer point transform of $\GT(\lambda)$ projects to the Schur function $s_{\lambda}$.  The latter beautiful result generalizes to Minkowski sums of Gelfand-Tsetlin polytopes and certain Schubert polynomials as well \cite{paper1}. In this paper we will be interested in the Gelfand-Tsetlin polytope $\GT(\lambda)$ from a purely discrete geometric point of view: we will explore it as a marked order polytope and as a flow polytope.

Ardila et.\ al introduced marked order polytopes and showed that Gelfand-Tsetlin polytopes are  examples of them in \cite{GTFFLVmarked};   the present authors have recently shown that   Gelfand-Tsetlin polytopes  are  flow polytopes  \cite{paper1}. Theorem \ref{thm:gtvol} summarizes previous work by Postnikov \cite[Theorem 15.1]{beyond} on the Gelfand-Tsetlin polytope and also demonstrates how the Gelfand-Tsetlin polytopes  being flow polytopes allows us to write the number of integer points and the  volume of $\GT(\lambda)$---which are equal respectively to the  dimension of the irreducible representation $V_{\lambda}$ of $GL(n)$ and the top homogeneous component of the dimension when viewed as a polynomial in $\lambda_1,\dots,\lambda_n$---in terms of Kostant partition functions:  

\begin{theorem} \label{thm:gtvol}
	Let $\lambda\in\mathbb{Z}^n_{\geq 0}$ be a partition and $r=|V(G_\lambda)|-1=\binom{n+2}{2}-3$. The volume of $\mathrm{GT(\lambda)}$ is given by
	\begin{align}
		\label{eqn:weylvol}
		\mathrm{Vol}\mbox{ }\mathrm{GT}(\lambda) &= \prod_{1\leq i<j\leq n} \frac{\lambda_i-\lambda_j}{j-i}\\ \label{eqn:shiftedsytvol}
		&=\sum_{b_1,\ldots,b_{n-1}\geq 0} \frac{(\lambda_1-\lambda_2)^{b_1}}{b_1!}\cdots \frac{(\lambda_{n-1}-\lambda_n)^{b_{n-1}}}{b_{n-1}!}N(b_1,\ldots,b_{n-1})\\
		\label{eqn:lidskiivol}
		&=\sum_{j}
		\frac{(\lambda_1-\lambda_2)^{j_1}}{j_1!}\cdots \frac{(\lambda_{n-1}-\lambda_n)^{j_{n-1}}}{j_{n-1}!} 
%		\frac{0^{j_n}}{j_n!}\cdots\frac{0^{j_r}}{j_r!}
		K_{G_\lambda}\left(j_1-1, \ldots, j_{n-1}-1,-1,\ldots,-1,0,\ldots,0,0\right).
	\end{align}
	
	\noindent The integer point count of $GT(\lambda)$ is given by 
	\begin{align}
		\label{eqn:weylintpts} 
		|{\rm GT}(\lambda)\cap \mathbb{Z}^{n+1 \choose 2}|&=\prod_{1\leq i<j\leq n}\frac{\lambda_i-\lambda_j+j-i}{j-i}\\
		\label{eqn:lidskiiintpts}
		&= \sum_{{j}}
		\binom{\lambda_1-\lambda_2+1}{j_1}\cdots
		\binom{\lambda_{n-1}-\lambda_n+1}{j_{n-1}}
		\binom{1}{j_n}\cdots\binom{1}{j_{\binom{n}{2}}}\binom{0}{j_{\binom{n}{2}+1}}\cdots\binom{0}{j_{r}}\\
		&\quad\quad\quad \cdot  K_{G_\lambda}\left(j_1-1, \ldots, j_{n-1}-1,j_n-1,\ldots,j_{\binom{n}{2}}-1,j_{\binom{n}{2}+1},\ldots,j_r,0\right) \notag.
%		&= \sum_{{j}}
%		\multiset{a_1-\i_1}{j_1}\cdots
%		\multiset{a_{n}-\i_n}{j_{n}} \cdot  K_{G_\lambda}\left(j_1-1, \ldots, j_{n-1}-1,-1,\ldots,-1,0,\ldots,0,0\right).  
	\end{align}
\end{theorem}

It is equalities  \eqref{eqn:lidskiivol} and   \eqref{eqn:lidskiiintpts} that  follow from $GT(\lambda)$ being a flow polytope. The other equations are known and follow from the representation theory of $GL(n)$ and from Postnikov's work \cite[Theorem 15.1]{beyond}. For the notation used in Theorem \ref{thm:gtvol}, we refer the reader to Section \ref{sec:gtflowpolytope}. We remark that from equations \eqref{eqn:shiftedsytvol} and \eqref{eqn:lidskiivol}, we obtain that the evaluations $N(b_1,\ldots,b_{n-1})$ and $K_{G_\lambda}\left(j_1-1, \ldots, j_{n-1}-1,-1,\ldots,-1,0,\ldots,0,0\right)$ are equal. We additionally provide a bijective proof of this in Section \ref{sec:gtflowpolytope}.

Section \ref{sec:marked} is devoted to  marked order polytopes in general.  
In light of the work of the second author with Morales and Striker \cite{floworderpolytopes} where they show that order polytopes of strongly planar posets are flow polytopes, it is natural to wonder if the   
 Gelfand-Tsetlin polytopes being both a marked order polytope and a flow polytope is part of a larger picture. Indeed, we show that marked order polytopes of strongly planar posets with certain conditions on the markings  are flow polytopes:

\begin{theorem} 
	\label{thm:markedorderflow} 
	If $(P,A,\lambda)$ is a marked poset admitting a bounded strongly planar embedding, then the marked order polytope $\mathcal{O}(P,A)_{\lambda}$ is integrally equivalent to the flow polytope $\F_{G_{(P,A,\lambda)}}$.
\end{theorem}
 
For the terminology used in Theorem \ref{thm:markedorderflow}, see Sections \ref{sec:marked} and \ref{sec:markedasflow}.  There is a natural way of subdividing the marked order polytope $\mathcal{O}(P,A)_{\lambda}$ into products of simplices labeled by certain linear extensions of the poset (Theorem \ref{thm:markedordervolume}), and there is a natural way of subdividing the flow polytope $\F_{G_{(P,A,\lambda)}}$ into products of simplices labeled by integer points of other flow polytopes (Section \ref{subsec:flow}). In Section \ref{sec:bijection} we show that these subdivisions map to each other under the integral equivalence of Theorem  \ref{thm:markedorderflow}, and we conclude by  bijecting their combinatorial labelings  in Corollary \ref{cor:generalbij}.

\medskip

\noindent {\bf Roadmap of the paper.}  In Section \ref{sec:gtflowpolytope} we define flow polytopes and show several consequences of Gelfand-Tsetlin polytopes being integrally equivalent to flow polytopes for their volume and Ehrhart polynomial formulas. It is well-known that the Gelfand-Tsetlin polytope is also a marked order polytope, and in Section \ref{sec:marked} we define marked order polytopes as well as collect and extend some known results about them. Section \ref{sec:markedasflow} proves Theorem \ref{thm:markedorderflow}, which gives conditions under which marked order polytopes are integrally equivalent to flow polytopes. The Gelfand-Tsetlin polytopes appear as a special case in this more general theory. Finally, in Section \ref{sec:bijection} we review the subdivision methods for flow polytopes and (marked) order polytopes, and we show that they map to each  other under the integral equivalence of Theorem  \ref{thm:markedorderflow}. We conclude by bijecting the two sets of combinatorial labels coming from the subdivisions of flow and marked order polytopes in Corollary \ref{cor:generalbij}.

\section{Gelfand-Tsetlin polytopes as flow polytopes}
\label{sec:gtflowpolytope}

In this section we recall the result from \cite{paper1} that the Gelfand-Tsetlin polytope is integrally equivalent to  a flow polytope, and then we study its volume and Ehrhart polynomial.  We start by defining flow polytopes and providing the necessary background on them following \cite{paper1}. 

 \subsection{Background on flow polytopes}

Let $G$ be a loopless directed acyclic connected (multi-)graph on the vertex set $[n+1]$ with $m$ edges. An integer vector $a=(a_1,\ldots,a_n,-\sum_{i=1}^na_i)\in \mathbb{Z}^{n+1}$ is called a \textbf{netflow vector}. A pair $(G,a)$ will be referred to as a \textbf{flow network}. To minimize notational complexity, we will typically omit the netflow $a$ when referring to a flow network $G$, describing it only when defining $G$. When not explicitly stated, we will always assume vertices of $G$ are labeled so that $(i,j)\in E(G)$ implies $i<j$.

To each edge $(i,j)$ of $G$, associate the type $A$ positive root $e_i-e_j\in\mathbb{R}^n$. Let $M_G$ be the incidence matrix of $G$, the matrix whose columns are the multiset of vectors $e_i-e_j$ for $(i,j)\in E(G)$. A \textbf{flow} on a flow network $G$ with netflow $a$ is a vector $f=(f(e))_{e\in E(G)}$ in $\mathbb{R}_{\geq 0}^{E(G)}$ such that $M_Gf=a$. Equivalently, for all $1\leq i \leq n$, we have 
\[\sum_{e=(k,i)\in E(G)}f(e)+a_i = \sum_{e=(i,k)\in E(G)} f(e). \]
The fact that the netflow of vertex $n+1$ is $-\sum_{i=1}^n a_i$ is implied by these equations.

Define the \textbf{flow polytope} $\mathcal{F}_G(a)$ of a graph $G$ with netflow $a$ to be the set of all flows on $G$:
\[\mathcal{F}_G=\mathcal{F}_G(a)=\{f\in\mathbb{R}^{E(G)}_{\geq 0} \mid M_Gf=a \}. \]

Given a graph $G$, the \textbf{Kostant partition function} $K_G$ of $G$ evaluated at a vector $b\in\mathbb{Z}^{n+1}$ is the number of ways to write $b$ as a nonnegative integer combination of the multiset of vectors $\{e_i-e_j\mid (i,j)\in E(G) \}$, or equivalently
\[K_G(b) = \left|\mathcal{F}_G(b)\cap\mathbb{Z}^{E(G)}\right|. \]

\begin{remark}
	\label{rem:flownetwork}
	When $G$ is a flow network $(G,a)$, we will write $\mathcal{F}_G$ for $\mathcal{F}_G(a)$. For any $b\in \mathbb{Z}^{n+1}$, we will write $\mathcal{F}_G(b)$ and $K_G(b)$  when we wish to use a vector possibly different from the netflow $a$ associated to $G$.
\end{remark}

The following remarkable theorem gives the volume and Ehrhart polynomial formulas for a family of flow polytopes.

\begin{theorem}[Baldoni--Vergne--Lidskii formulas {\cite[Thm.\ 38]{BV2}}]
\label{thm:lidskii} 
Let $G$ be a connected graph on the vertex set $[n+1]$ with $m$
edges, and let $a=(a_1,\ldots,a_n,-\sum_{i=1}^n a_i)$ with $a_i \in \mathbb{Z}_{\geq 0}$ for $i\in [n]$.  Direct the edges of $G$ by $i\to j$ if $i<j$, and assume there is at least one outgoing edge at vertex $i$ for each $i=1,\ldots,n$. Then
\begin{align} \label{eq:lidskiivolume}
\mathrm{Vol}\mbox{ } \mathcal{F}_G(a) &= \sum_{j}
\frac{a_1^{j_1}}{j_1!}\cdots \frac{a_n^{j_n}}{j_n!} K_{G}\left(j_1-\om_1, \ldots, j_n - \om_n,0\right),\\
|\F_G(a) \cap \mathbb{Z}^{E(G)}|&= \sum_{{j}}
\binom{a_1+\om_1}{j_1}\cdots
\binom{a_{n}+\om_n}{j_{n}} \cdot  K_{G}\left(j_1-\om_1, \ldots, j_n -
                 \om_n,0\right), \label{eq:lidskiikost}\\
&= \sum_{{j}}
\multiset{a_1-\i_1}{j_1}\cdots
\multiset{a_{n}-\i_n}{j_{n}} \cdot  K_{G}\left(j_1-\om_1, \ldots, j_n - \om_n,0\right), \label{eq:lidskiikostmultiset}    
\end{align} 
for $\om_i=\out_i-1$ and $\i_i=\ind_i-1$, where $\out_i$ and $\ind_i$ denote the outdegree and indegree of vertex $i$ in $G$. Each sum is over weak compositions $j=(j_1,j_2,\ldots,j_n)$ of $m-n$ that are $\geq (\om_1,\ldots,\om_n)$ in dominance order, and $\multiset{n}{k}=\binom{n+k-1}{k}$.
\end{theorem}

\subsection{The Gelfand-Tsetlin polytope as a flow polytope} 

The following theorem was proved in \cite{paper1}, where $G_\lambda$ is a network to be defined below.

\begin{theorem}[\cite{paper1}]
	\label{thm:gtflowpolytope}
	$\mathrm{GT}(\lambda)$ is integrally equivalent to $\mathcal{F}_{G_\lambda}$.
\end{theorem}

Recall that two integral polytopes $\mathcal{P}$ in $\mathbb{R}^d$ and $\mathcal{Q}$ in $\mathbb{R}^m$ are \textbf{integrally equivalent}
if there is an affine transformation
$\varphi\colon\mathbb{R}^d \to \mathbb{R}^m$ whose restriction to
$\mathcal{P}$ is a bijection $\varphi\colon \mathcal{P} \to \mathcal{Q}$
that preserves the lattice, i.e., $\varphi$ is a
bijection between $\mathbb{Z}^d \cap \aff(\mathcal{P})$ and
$\mathbb{Z}^m \cap \aff(\mathcal{Q})$, where $\aff(\cdot)$ denotes affine span. The map $\varphi$ is called an  \textbf{integral equivalence}. Note that integrally equivalent polytopes have the same Ehrhart polynomials, and therefore the same volume.    

\medskip
 We now define the flow network $G_\lambda$, describing the graph and its associated netflow (see Remark \ref{rem:flownetwork}). For an illustration of $G_{\lambda}$, see  Figure \ref{Glambda}. 
\begin{definition}
	\label{def:Glambda}
	For a partition $\lambda\in\mathbb{Z}^n_{\geq 0}$ with $n\geq 2$, let $G_\lambda$ be defined as follows:
	
	\noindent If $n=1$, let $G_\lambda$ be a single vertex $v_{22}$ defined to have flow polytope consisting of one point, $0$. Otherwise, let $G_\lambda$ have vertices 
	\[V(G_\lambda) = \{v_{ij} \mid  2\leq i\leq j \leq n\}\cup\{v_{i,i-1} \mid  3\leq i\leq n+2\}\cup\{v_{i,n+1} \mid  3\leq i \leq n+1\} \]
	and edges 
	\begin{equation*}
		\begin{split}
		E(G_\lambda) &= \{(v_{ij},v_{i+1,j}) \mid 2\leq i\leq j\leq n \}\cup\{(v_{i,n+1},v_{i+1,n+1}) \mid 3\leq i \leq n+1 \}\\
		&\quad\cup\{(v_{ij},v_{i+1,j+1}) \mid  2\leq i\leq j\leq n \}\cup\{(v_{i,i-1},v_{i+1,i}) \mid  3\leq i \leq n+1 \}.
		\end{split}
	\end{equation*}
	The default netflow vector on $G_\lambda$ is as follows:
	\begin{itemize}
		\item To vertex $v_{2j}$ for $2\leq j \leq n$, assign netflow $\lambda_{j-1}-\lambda_{j}$.
		\item To vertex $v_{n+2,n+1}$, assign netflow $\lambda_{n}-\lambda_{1}$.
		\item To all other vertices, assign netflow $0$.
	\end{itemize}
	Given a flow on $G_\lambda$, denote the flow value on each edge $(v_{ij},v_{i+1,j})$ by $a_{ij}$, and denote the flow value on each edge $(v_{ij},v_{i+1,j+1})$ by $b_{ij}$.
\end{definition}

\begin{figure}[ht]
	\includegraphics[scale=.45]{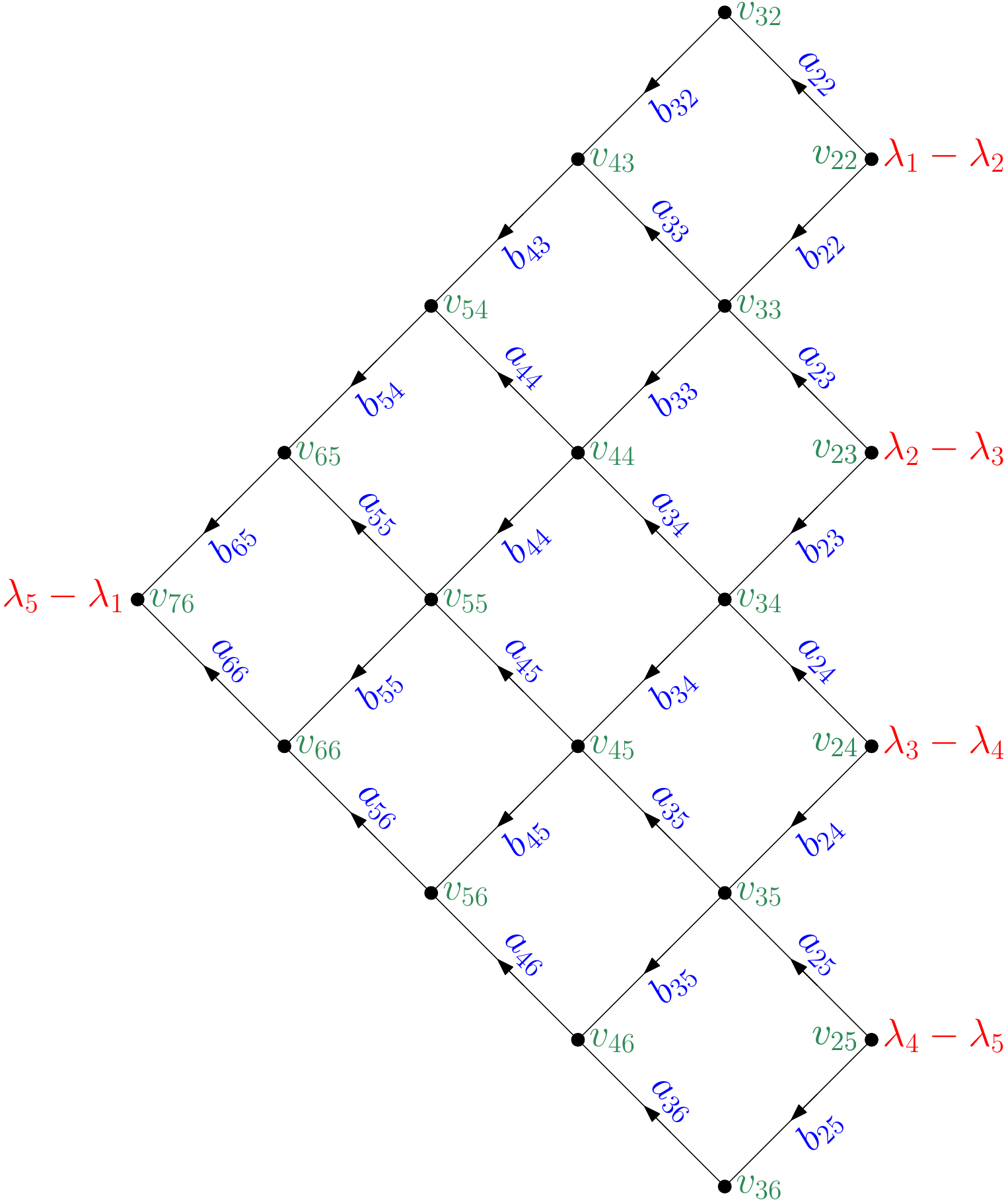}
	\caption{The flow network $G_\lambda$ with $\ell(\lambda)=5$.}
	\label{Glambda}
\end{figure}
  
  We note that viewing $\mathrm{GT}(\lambda)$ as a marked order polytope (\cite{GTFFLVmarked}), Theorem \ref{thm:gtflowpolytope} is also a special case of our more general Theorem \ref{thm:markedorderflow}.  
   
Several expressions for the volume and integer point count of $\mathrm{GT}(\lambda)$ are given in the following theorem. To  apply the Lidskii formulas to $G_\lambda$, we list the vertices of $G_\lambda$ in the following order: First are the vertices $\{v_{ij} \mid  2\leq i\leq j\leq n\}$ ordered lexicographically; next, the vertices $\{v_{i,i-1} \mid  3\leq i\leq n+1\}$ ordered lexicographically; then, the vertices $\{v_{i,n+1} \mid  3\leq i \leq n+1\}$ ordered lexicographically; and lastly, the sink vertex $v_{n+2,n+1}$. 

In (\ref{eqn:shiftedsytvol}) below, we need a few definitions: 
A shifted standard Young tableaux (shSYT) is a bijection 
$T\colon \left\{(i,j) \mid 1\leq i\leq j\leq n\right\} \to \left\{1,2,\ldots,\binom{n+1}{2} \right\}$
such that $T(i,j)<T(i+1,j)$ and $T(i,j)<T(i,j+1)$. The diagonal vector of a shSYT $T$ is $(T(1,1),T(2,2),\ldots,T(n,n)).$ Denote by $N(b_1,\ldots,b_{n-1})$ the number of shSYT $T$ with diagonal entries $T(i,i)=i+b_{1}+\cdots+b_{i-1}$.
\begin{namedtheorem}[\ref{thm:gtvol}]
	Let $\lambda\in\mathbb{Z}^n_{\geq 0}$ be a partition and $r=|V(G_\lambda)|-1=\binom{n+2}{2}-3$. The volume of $\mathrm{GT(\lambda)}$ is given by
	\begin{align*}
		\mathrm{Vol}\mbox{ }\mathrm{GT}(\lambda) &= \prod_{1\leq i<j\leq n} \frac{\lambda_i-\lambda_j}{j-i} \tag{\ref{eqn:weylvol}}\\  
		&=\sum_{b_1,\ldots,b_{n-1}\geq 0} \frac{(\lambda_1-\lambda_2)^{b_1}}{b_1!}\cdots \frac{(\lambda_{n-1}-\lambda_n)^{b_{n-1}}}{b_{n-1}!}N(b_1,\ldots,b_{n-1})\tag{\ref{eqn:shiftedsytvol}}\\
			&=\sum_{j}
		\frac{(\lambda_1-\lambda_2)^{j_1}}{j_1!}\cdots \frac{(\lambda_{n-1}-\lambda_n)^{j_{n-1}}}{j_{n-1}!} 
%		\frac{0^{j_n}}{j_n!}\cdots\frac{0^{j_r}}{j_r!}
		K_{G_\lambda}\left(j_1-1, \ldots, j_{n-1}-1,-1,\ldots,-1,0,\ldots,0,0\right).\tag{\ref{eqn:lidskiivol}}
	\end{align*}
	
	\noindent The integer point count of $GT(\lambda)$ is given by 
	\begin{align*}
 		|{\rm GT}(\lambda)\cap \mathbb{Z}^{n+1 \choose 2}|&=\prod_{1\leq i<j\leq n}\frac{\lambda_i-\lambda_j+j-i}{j-i}\tag{\ref{eqn:weylintpts}}\\
		&= \sum_{{j}}
		\binom{\lambda_1-\lambda_2+1}{j_1}\cdots
		\binom{\lambda_{n-1}-\lambda_n+1}{j_{n-1}}
		\binom{1}{j_n}\cdots\binom{1}{j_{\binom{n}{2}}}\binom{0}{j_{\binom{n}{2}+1}}\cdots\binom{0}{j_{r}}\tag{\ref{eqn:lidskiiintpts}}\\
		&\quad\quad\quad \cdot  K_{G_\lambda}\left(j_1-1, \ldots, j_{n-1}-1,j_n-1,\ldots,j_{\binom{n}{2}}-1,j_{\binom{n}{2}+1},\ldots,j_r,0\right).
%		&= \sum_{{j}}
%		\multiset{a_1-\i_1}{j_1}\cdots
%		\multiset{a_{n}-\i_n}{j_{n}} \cdot  K_{G_\lambda}\left(j_1-1, \ldots, j_{n-1}-1,-1,\ldots,-1,0,\ldots,0,0\right).  
	\end{align*}
\end{namedtheorem}
\begin{proof}
	In \cite{beyond}, (\ref{eqn:weylvol}), (\ref{eqn:shiftedsytvol}), and (\ref{eqn:weylintpts}) are shown. Applying Theorem \ref{thm:lidskii} to the flow network $G_\lambda$ yields (\ref{eqn:lidskiivol}) and (\ref{eqn:lidskiiintpts}).
\end{proof}

\begin{corollary} 
	\label{cor:bij}
	Comparing the volume formulas (\ref{eqn:shiftedsytvol}) and (\ref{eqn:lidskiivol}), we obtain that
	\[
		N(b_1,\ldots,b_{n-1})=K_{G_\lambda}\left(b_1-1, \ldots, b_{n-1}-1,-1,\ldots,-1,0,\ldots,0,0\right)
	\]
	for all $b_1,\ldots,b_{n-1}\geq 0$.
\end{corollary}

One can also see Corollary~\ref{cor:bij} bijectively as follows. Given a shSYT $T$ counted by $N(b_1, \dots, b_{n-1})$, 
 define for $1 \leq i < j \leq n$:
\begin{align*}
a_{j-i+1,j} &= \left |\{(i',j') \mid T(i,j-1) < T(i',j') < T(i,j), \quad i' < i, \quad j' \geq j\}\right|,\\
b_{j-i+1,j} &= \left |\{(i',j') \mid T(i,j) < T(i',j') < T(i+1,j), \quad i' \leq i, \quad j' > j\}\right|.
\end{align*}
We claim that these define a flow on $G_\lambda$ with netflow $\left(b_1-1, \ldots, b_{n-1}-1,-1,\ldots,-1,0,\ldots,0,0\right)$.

\begin{itemize}
\item Note that $a_{ij} = 0$ for $i=j$, while $b_{ij} = 0$ for $j=n$. Thus there is no netflow at vertex $v_{ij}$ unless $2 \leq i \leq j \leq n$.
\item At $v_{2j}$, the netflow is $a_{2j}+b_{2j}$, which counts pairs $(i',j')$ such that $T(j-1,j-1)<T(i',j') < T(j-1,j)$ or $T(j-1,j)<T(i',j')<T(j,j)$, of which there are $T(j,j) - T(j-1,j-1) - 2 = b_{j-1}-1$.
\item At any other vertex $v_{ij}$, the netflow is $a_{ij} + b_{ij} - a_{i-1,j} - b_{i-1,j-1}$. But $a_{ij} + b_{ij}$ and $a_{i-1,j} + b_{i-1,j-1}$ both count pairs $(i',j')$ such that $T(j-i+1,j-1) < T(i',j') < T(j-i+2,j)$ with the only difference being that $(i',j') = (j-i+1,j)$ is not counted by the first quantity but is counted by the second. It follows that the netflow is $-1$, as desired. 
\end{itemize}

For the inverse map, we can construct the shSYT $T$ inductively: by removing the vertices $v_{2j}$ and edges with flows $a_{2j}$ and $b_{2j}$ from $G_\lambda$, we arrive at the graph for a partition of length $n-1$ with netflow
\[(b_{22}+a_{23}-1, b_{23} + a_{24} - 1, \dots, b_{2,n-1}+a_{2n}-1, -1, \dots, -1, 0, \dots, 0).\]
By induction, we can construct from this a shSYT $T'$ with side length $n-1$ whose $i$th diagonal entry $T'(i,i)$ is
\begin{align*}
i + (b_{22} + a_{23}) + (b_{23} + a_{24}) + \cdots + (b_{2i}+a_{2,i+1})
&= i + (b_1-1) + (b_2-1) + \cdots + (b_{i-1}-1) + a_{2,i+1}\\
&= 1+b_1 + b_2 + \cdots + b_{i-1} + a_{2,i+1}.
\end{align*}
Hence \[1 + b_1 + b_2 + \cdots + b_{i-1} \leq T'(i,i) \leq b_1 + b_2 + \cdots + b_i.\]
Then let us modify $T'$ by adding $1$ to the entries $1, \dots, b_1$, adding $2$ to the entries $1 + b_1, \dots, b_1 + b_2$, and so forth, which in particular will add $i$ to $T'(i,i)$. We can then attach to $T'$ a new diagonal with entries $1$, $2+b_1$, $3+b_1+b_2$, \dots, which will yield a shSYT $T$ of side length $n$ with the desired diagonal entries. It is straightforward to check that these two maps are inverses of one another, completing the bijection.

We will provide a bijective proof of a generalization of Corollary \ref{cor:bij} in Section \ref{sec:bijection}, in the more general setting of strongly planar marked order polytopes.

 \section{Marked Order Polytopes}
\label{sec:marked}

In \cite{GTFFLVmarked}, Ardila, Bliem, and Salazar observed that the Gelfand-Tsetlin polytope is a section of an order polytope. Inspired, they introduced marked posets and marked poset polytopes, generalizing Stanley's notion of order and chain polytopes introduced in \cite{orderchainpolytopes}. In this section we give background on unmarked and marked order polytopes, and we explain the generalizations of several results from order polytopes to marked order polytopes.

\begin{definition}
	A \textbf{marked poset} $(P,A,\lambda)$ consists of a finite poset $P$, a subposet $A \subseteq P$ containing all its extremal elements, and an order-preserving map $\lambda\colon A\to\mathbb{R}$.
	We identify $(P,A,\lambda)$ with the \textbf{marked Hasse diagram} in which we label the elements $a \in A$ with $\lambda(a)$ in the Hasse diagram of $P$.
\end{definition}
 
\begin{definition} The \textbf{marked order polytope} of $(P,A,\lambda)$ is
	\[\mathcal{O}(P,A)_\lambda = \{x \in \mathbb{R}^{P} \mid
	x_p \leq x_q \textrm{ for } p < q \mbox{ in $P$ and } x_a=\lambda(a) \mbox{ for } a\in A\}.\] 
	Let $\widetilde{\mathcal{O}}(P,A)_\lambda$ denote $\mathcal{O}(P,A)_\lambda$ projected onto the coordinates $P\backslash A$.
\end{definition}

Stanley's construction of the \textbf{order polytope} $\mathcal{O}(P)$  \cite{orderchainpolytopes} is a special case of a marked order polytope. Given a finite poset $P$, add a new smallest and largest element to obtain $\widehat{P} = P \sqcup \{ \hat 0, \hat 1 \}$. Let $A = \{\hat 0, \hat 1 \}$ and $\lambda(\hat a)=a$. Then 
\[\mathcal{O}(P) = \widetilde{\mathcal{O}}(\widehat{P}, A)_{\lambda}.\]
 
In general, computing or finding a combinatorial interpretation for the volume of a polytope is a hard problem. Order polytopes are an especially nice class of polytopes whose volume has a combinatorial interpretation.

\begin{theorem}[Stanley \cite{orderchainpolytopes}] \label{thm:stop}
Given a poset $P$, we have that 
\begin{itemize}
\item[(i)] the vertices of $\mathcal{O}(P)$  are in bijection with
  characteristic functions of complements of order ideals of $P$, 
\item[(ii)]  the normalized volume of $\mathcal{O}(P)$  is $e(P)$, where
  $e(P)$ is the number of linear extensions of $P$, and
\item[(iii)] the Ehrhart polynomial $L_{\mathcal{O}(P)}(m)$ of
  $\mathcal{O}(P)$ equals the order polynomial $\Omega(P,m+1)$ of $P$.
\end{itemize}
\end{theorem}
 
 We now explain how Theorem \ref{thm:stop}  generalizes to the setting of marked order polytopes.  

For part (i), the vertices and facial structure of marked order polytopes are described by Pegel in \cite{markedfacestructure}. A point $x\in\mathcal{O}(P,A)_{\lambda}$ induces a partition $\pi_x$ of $P$ that is the transitive closure of the relation $p_1\sim_x p_2$ if $x_p=x_q$ and $p,q$ are comparable. A point $x\in\mathcal{O}(P,A)_\lambda$ is a vertex if and only if each block of $\pi_x$ contains a marked point. In the case of an unmarked order polytope $O(P)$, the blocks will be an order ideal and its complement, so the vertices are characteristic functions.

Part (ii) of  Theorem \ref{thm:stop} has a beautiful geometric justification: order polytopes have a canonical subdivision into $e(P)$ unimodular simplices. Consider $\mathcal{O}(P)$ cut with all hyperplanes of the form $x_p=x_q$ where $p,q\in P$ with $p$ and $q$ incomparable. The regions of this arrangement correspond to the ways of totally ordering the coordinates $x_p$ compatible with all inequalities of $\mathcal{O}(P)$, that is, linear extensions of $P$. Each region is defined by inequalities of the form $y_1\leq y_2\leq\cdots\leq y_{|P|}$ for $y_1,\ldots,y_{|P|}$ a permutation of $\{x_p\}_{p\in P}$, so each region is a simplex.

The following theorem generalizes part (ii) of  Theorem \ref{thm:stop} to marked order polytopes. For notational convenience, we will take a \textbf{linear extension} of a poset $P$ with $n$ elements to be an \textit{order-reversing} bijection $\sigma\colon[n]\to P$, so for example $\sigma(1)$ will be a maximal element of $P$. We will generally label the elements of $A$ as $\{p_1, p_2, \dots, p_k\}$ such that $\lambda(p_1) \geq \cdots \geq \lambda(p_k)$ (and, additionally, if $p_i > p_j$ in $P$, then $i < j$).

 \begin{theorem}  
	\label{thm:markedordervolume} (cf.  \cite[Theorem 3.2]{two})
	If $(P,A,\lambda)$ is a marked poset with marked elements $A=\{p_1,\ldots,p_k \}$ having markings $\lambda(p_1)\geq\cdots\geq\lambda(p_k)$ denoted $\lambda_1,\ldots,\lambda_k$, then
	\[\mathrm{Vol}\mbox{ }\widetilde{\mathcal{O}}(P,A)_{\lambda} = \sum_{a_{1},\dots,a_{k-1}\ge0}N_{P, A, \lambda}(a_{1},\dots,a_{k-1})\frac{(\lambda_{1}-\lambda_{2})^{a_{1}}}{a_{1}!}\cdot\cdots\cdot\frac{(\lambda_{k-1}-
		\lambda_{k})^{a_{k-1}}}{a_{k-1}!},\]
	where $N_{P,A,\lambda}(a_{1},\dots, a_{k-1})$ is the number of linear extensions of $P$ such that elements of $A=\{p_1, \ldots, p_k\}$  occur at positions $1,2+a_1,\ldots,k+a_1+\cdots+a_{k-1}$, respectively.
\end{theorem}

We note that when the markings $A$ are along a chain in the poset $P$, Stanley has shown the above theorem in his proof of a certain log-concavity conjecture which we explain below; see the proof of \cite[Theorem 3.2]{two}.  His proof can be generalized to the above setting. We provide the proof here for completeness and take a slightly different perspective via hyperplane cuts, much like Postnikov does in \cite{beyond}  for $\mathrm{GT}(\lambda)$.

\bigskip

\begin{proof}[Proof of Theorem \ref{thm:markedordervolume}.]
	Consider $\widetilde{\mathcal{O}}(P,A)_{\lambda}$ cut with all hyperplanes of the form $x_p=x_q$ or $x_p=\lambda(a)$, where $p,q\in P\backslash A$ with $p$ is incomparable with $q$, and $a\in A$ is incomparable with $p$. The regions of this arrangement correspond to the ways of totally ordering the coordinates $(x_p)_{p\in P}$ compatible with all inequalities and markings, that is, certain linear extensions of $P$. Let $\sigma\colon[n]\to P$ be a linear extension of $P$, say with $\sigma(i_j)=p_j$ for $j\in[k]$, $i_1<i_2<\cdots<i_k$. Since $A$ contains all minimal and maximal elements of $P$, note that $i_1=1$ and $i_k=n$. The associated region $\widetilde{\Pi}_{\sigma}$ in the subdivision is the projection of the region 
	\[
	\Pi_{\sigma}=\{x\in\mathbb{R}^P \mid  x_{\sigma(1)}\geq x_{\sigma(2)}\geq \dots\geq x_{\sigma(n)},\ x_{\sigma(i_{j})}=\lambda_{j}.\}
	\]
	onto the coordinates $\mathbb{R}^{P\backslash A}$. If nonempty, $\widetilde{\Pi}_\sigma$ is the direct product
	\[\widetilde{\Pi}_{\sigma}= 
	\prod_{j=1}^{k-1}\{\lambda_{j}\geq x_{\sigma(i_j+1)}\geq\cdots\geq x_{\sigma(i_{j+1}-1)}\geq \lambda_{j+1}\}\]
	where each term $\{\lambda_{j}\geq x_{\sigma(i_j+1)}\geq \cdots\geq x_{\sigma(i_{j+1}-1)}\geq \lambda_{j+1}\}$ is an $(i_{j+1}-i_j-1)$-dimensional simplex with volume $\frac{(\lambda_j-\lambda_{j+1})^{i_{j+1}-i_{j}-1}}{(i_{j+1}-i_{j}-1)!}$. Thus 
	\[
	\mbox{Vol }\widetilde{\Pi}_{\sigma}=\frac{(\lambda_{1}-\lambda_{2})^{a_{1}}}{a_{1}!}\cdots\frac{(\lambda_{k-1}-
		\lambda_{k})^{a_{k-1}}}{a_{k-1}!},\] where we set $a_j\coloneqq i_{j+1}-i_{j}-1$. Summing over all linear extensions $\sigma$ of $P$, we obtain
	\begin{align*}
	\mathrm{Vol}\mbox{ }\widetilde{\mathcal{O}}(P,A)_{\lambda} & = \sum_{\sigma \in \mathcal{L}(P)}\mathrm{Vol}\mbox{ }\widetilde{\Pi}_{\sigma}\\
	& = \sum_{a_{1},\dots,a_{k-1}\ge0}N_{P, A, \lambda}(a_{1},\dots,a_{k-1})\frac{(\lambda_{1}-\lambda_{2})^{a_{1}}}{a_{1}!}\cdots\frac{(\lambda_{k-1}-
		\lambda_{k})^{a_{k-1}}}{a_{k-1}!}.\qedhere
	\end{align*}
\end{proof}

 Marked order polytopes also enjoy a Minkowski sum property and decomposition.
\begin{theorem}[\cite{chainorderpolytopes}]
	Let $P$ be a poset and $A$ a subposet. If $\lambda,\mu\colon A\to\mathbb{R}$ are markings, then \[\mathcal{O}(P,A)_{\lambda+\mu} =\mathcal{O}(P,A)_{\lambda}+ \mathcal{O}(P,A)_{\mu}. \] 
\end{theorem}

\begin{corollary}
	\label{cor:markedorderminkowski}
	For $(P,A,\lambda)$ a marked poset with marked elements $A=\{p_1,\ldots,p_k \}$ having markings $\lambda(p_1)\geq \cdots\geq \lambda(p_k)$, let $\omega_i\colon A\to \mathbb{R}$ be the map such that $\omega_i(p_j)=1$ for $j\leq i$ and $\omega_i(p_{j})=0$ if $j>i$. Then, taking $\lambda(p_{k+1})$ to mean $0$, $\mathcal{O}(P,A)_\lambda$ decomposes into the Minkowski sum	
	\[\mathcal{O}(P,A)_{\lambda} = \sum_{i=1}^k (\lambda(p_i)-\lambda(p_{i+1}))\mathcal{O}(P,A)_{\omega_i}. \]
\end{corollary}

\subsection{A Log-Concavity Result}

Recall that a sequence $b_0,b_1,\ldots,b_m$ of non-negative real numbers is said to be \textbf{log-concave} if $b_i^2\geq b_{i-1}b_{i+1}$ for $1\leq i\leq m-1$. In particular, a log-concave sequence is \textbf{unimodal}, that is for some $j$, we have $b_1\leq b_2\leq\cdots\leq b_j$ and $b_j\geq b_{j+1}\geq\cdots\geq b_m$. 

Using the Alexandrov-Fenchel inequalities and the volume formula for order polytopes, Stanley proved the following log-concavity result in the special case where all marked elements of $P$ lie on a chain in \cite{two}.

\begin{theorem}
	\label{thm:logconcave}
	Let $(P,A,\lambda)$ be a marked poset with marked elements $A=\{p_1,\ldots,p_k\}$ having markings $\lambda(p_1)\geq\cdots\geq\lambda(p_k)$ denoted $\lambda_1,\ldots,\lambda_k$. If $a_1,\ldots,a_k\geq 0$ with $a_{j-1}>1$ and $a_j>1$ for some $j$, then
	\[N_{P, A, \lambda}(a_{1},\dots,a_{k-1})^2\geq N_{P, A, \lambda}(a_{1},\dots,a_{j-1},a_{j}-1,a_{j+1},\ldots,a_{k-1})
	N_{P, A, \lambda}(a_{1},\dots,a_{j-1},a_{j}+1,a_{j+1},\ldots,a_{k-1}).\]
\end{theorem}

Before proving Theorem \ref{thm:logconcave}, we give some background on the theory of mixed volumes and the Alexandrov-Fenchel inequalities following that of \cite{two}. If $K_1,\ldots, K_s$ are convex bodies (nonempty compact convex subsets) of $\mathbb{R}^n$, fix weights $r_1,\ldots,r_s\geq 0$ and let $K$ denote the Minkowski sum
\[K=r_1K_1+\cdots+r_sK_s=\{r_1t_1+\cdots+r_st_s \mid t_i\in K_i \}. \]
The volume $V(K)$ of $K$ is a homogeneous polynomial of degree $n$ in $r_1,\ldots,r_s$
\[V(K) = \sum_{i_1=1}^s\sum_{i_2=1}^s\cdots \sum_{i_n=1}^s V_{i_1,\ldots,i_n} r_{i_1}\cdots r_{i_n}. \]
The coefficients $V_{i_1,\ldots,i_n}$ are uniquely determined by requiring they be symmetric up to permutations of subscripts. The coefficient $V_{i_1,\ldots,i_n}$ depends only on $K_{i_1},\ldots,K_{i_n}$ and is called the \textbf{mixed volume} of $K_{i_1},\ldots,K_{i_n}$. If we write $V(K_1^{a_1},\ldots,K_{s}^{a_s})$ for 
\[V\left(\underbrace{K_1,\ldots,K_1}_{a_1}, \ldots , \underbrace{K_s,\ldots,K_s}_{a_s}\right), \]
then 
\[V(K)=\sum_{a_1+\cdots+a_s=n}\binom{n}{a_1,\ldots,a_s}V\left(K_1^{a_1},\ldots,K_s^{a_s}\right)r_1^{a_1}\cdots r_s^{a_s}. \]
The well-known result about mixed volumes needed for the proof of Theorem \ref{thm:logconcave} is the following.
 
\begin{theorem}[Alexandrov-Fenchel Inequalities, \cite{log-conc}]
	Given $0 \leq m \leq n$ and convex bodies $C_1,\ldots,C_{n-m},K,L\subseteq\mathbb{R}^n$, the sequence $(b_0,b_1,\ldots,b_m)$ defined by 
	\[b_i=V\left(C_1,\ldots,C_{n-m},K^{m-i},L^i\right) \]
	is log-concave.
\end{theorem}

We can now give the proof of Theorem \ref{thm:logconcave}.
\begin{proof}[Proof of Theorem \ref{thm:logconcave}]
	Corollary \ref{cor:markedorderminkowski} yields the Minkowski sum
	\[\widetilde{\mathcal{O}}(P,A)_{\lambda} = \sum_{i=1}^k (\lambda_i-\lambda_{i+1})\widetilde{\mathcal{O}}(P,A)_{\omega_i},\]
	so taking $\lambda_{k+1}=0$,
	\begin{align*}
	\mathrm{Vol}\mbox{ }\widetilde{\mathcal{O}}(P,A)_{\lambda} &= 
	\mathrm{Vol}\mbox{ }\sum_{i=1}^k (\lambda_i-\lambda_{i+1})\widetilde{\mathcal{O}}(P,A)_{\omega_i}\\
	&=\sum_{a_1+\cdots+a_{k-1}=|P|-k}\binom{|P|-k}{a_1,\ldots,a_{k-1}} V(\widetilde{\mathcal{O}}(P,A)^{a_1}_{\omega_1},\ldots,\widetilde{\mathcal{O}}(P,A)^{a_{k-1}}_{\omega_{k-1}})(\lambda_1-\lambda_2)^{a_1}\cdots(\lambda_{k-1}-\lambda_k)^{a_{k-1}}.
	\end{align*}
	Comparing this with the volume formula 
	\[\mathrm{Vol }\mbox{ }\widetilde{\mathcal{O}}(P,A)_{\lambda} = \sum_{a_{1},\dots,a_{k-1}\ge0}N_{P, A, \lambda}(a_{1},\dots,a_{k-1})\frac{(\lambda_{1}-\lambda_{2})^{a_{1}}}{a_{1}!}\cdot\cdots\cdot\frac{(\lambda_{k-1}-
		\lambda_{k})^{a_{k-1}}}{a_{k-1}!}\]
	of Theorem \ref{thm:markedordervolume}, we obtain
	\[N_{P, A, \lambda}(a_{1},\dots,a_{k-1}) = (|P|-k)!V(\widetilde{\mathcal{O}}(P,A)^{a_1}_{\omega_1},\ldots,\widetilde{\mathcal{O}}(P,A)^{a_{k-1}}_{\omega_{k-1}}). \]
	An application of the Alexandrov-Fenchel Inequality completes the proof.
\end{proof}

\section{Marked order polytopes as flow polytopes}
\label{sec:markedasflow}

In this section we prove that for strongly planar posets with special markings, the marked order polytopes are integrally equivalent to flow polytopes. This generalizes a result of M\'esz\'aros-Morales-Striker \cite[Theorem 3.14]{floworderpolytopes} for (unmarked) order polytopes, which we now review.

A poset $P$ is \textbf{strongly planar} if the Hasse diagram of $\widehat{P}\coloneqq P
\sqcup \{\hat{0}, \hat{1}\}$ is planar and can be drawn in the plane so that the  $y$-coordinates of vertices respect the order of $P$. When we refer to a \textbf{bounded embedding} of $P$, we will mean a strongly planar drawing of the Hasse diagram of $\widehat{P}$ with an additional two edges between $\hat{0}$ and  $\hat{1}$ added, one drawn to the left of $\widehat{P}$ and the other drawn to the right (see Figure \ref{fig:bspe}). We will view this embedding as a planar graph and discuss its (bounded) faces in the usual graph-theoretic sense.

\begin{figure}[ht]
	\includegraphics[scale=.8]{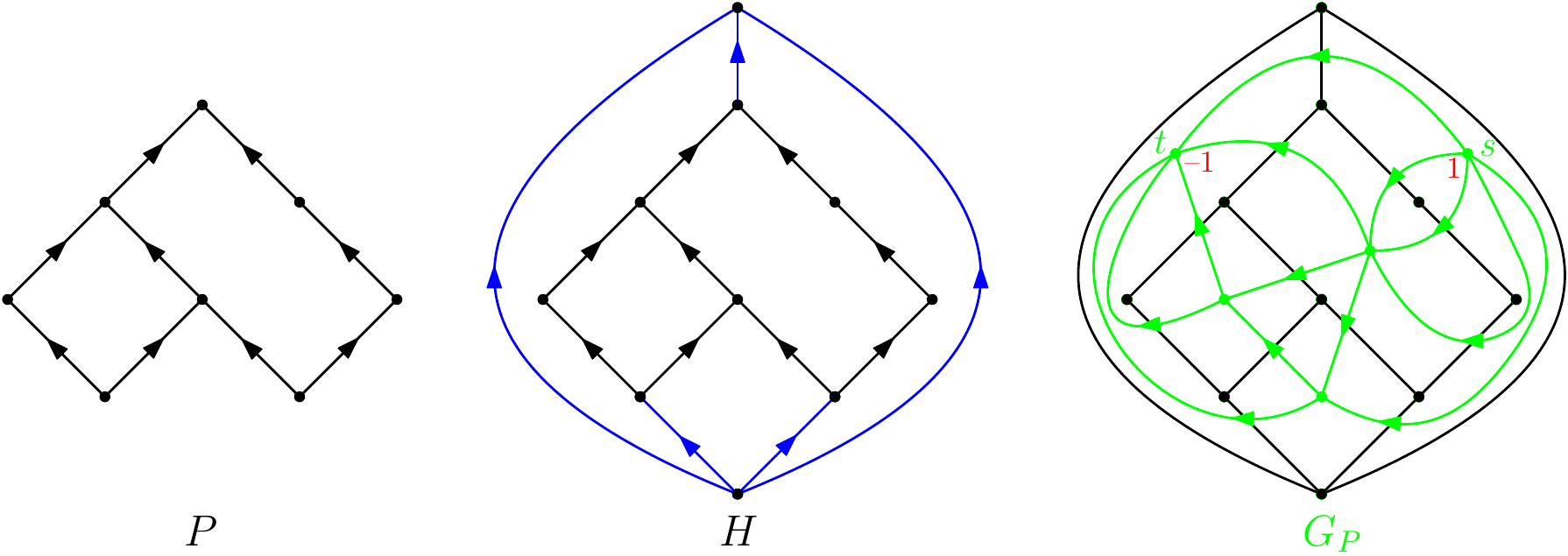}
	\caption{The Hasse diagram of a poset $P$ (left), a bounded embedding $H$ of $P$ (middle), and the directed graph $G_P$ drawn over $H$ (right).}
	\label{fig:bspe}
\end{figure}

We begin by recalling the case of order polytopes. 
Given a strongly planar poset $P$, let $H$ be a bounded embedding of $P$ (viewed as a planar graph). Let $H^*$ be the graph-theoretic dual of $H$. Define $G_P$ to be the subgraph of $H^*$ obtained by deleting the vertex corresponding to unbounded face of $H$. Denote the two vertices of $G$ that lie in faces of $H$ containing the edges $(\hat{0},\hat{1})$ by $s$ and $t$ with $s$ on the right and $t$ on the left.

Assign each edge $e$ in $G_P$ an orientation by the following rule: in the construction of $H^*$ the edge $e$ crosses an edge $p<q$ of $H$; orient $e$ so that while traversing $e$, $q$ is on the right and $p$ is on the left. Make $G_P$ into a flow network by assigning netflow $1$ to $s$, $-1$ to $t$, and $0$ to all other vertices.

\begin{theorem}[{\cite[Theorem 3.14]{floworderpolytopes}}]
	\label{thm:orderpolytopestronglyplanar}
	Let $P$ be a strongly planar poset and $G_P$ be the flow network constructed above. The polytopes $\mathcal{O}(P)$ and $\mathcal{F}_{G_P}$ are integrally equivalent.
\end{theorem}
\begin{proof}[Proof sketch]
	The map from $\mathcal{O}(P)\to\mathcal{F}_{G_P}$ is given by $(x_p)_{p\in P}\mapsto f$ where $f(e)=x_q-x_p$ if $e$ crosses the edge $p<q$ in $H$ and $x_{\hat{0}}$, $x_{\hat{1}}$ are taken to be 0 and 1 respectively. For the inverse, take a flow $f$ on $G_P$. For each $p\in P$, choose any path in $H$ from $\hat{0}$ to $p$. To define $x_p$, sum the flow values $f(e)$ on each edge $e\in G_P$ crossing an edge in the chosen path from $\hat{0}$ to $p$ in $H$.
\end{proof}

We now generalize Theorem \ref{thm:orderpolytopestronglyplanar} to marked order polytopes. We begin with some terminology used to define the marked analogue of a strongly planar poset. If $F$ is a bounded face of a bounded embedding $H$ of $P$, let $p$ denote the minimum element of $F$ and let $q$ denote the maximum. The graph $F\backslash\{p,q\}$ has two components whose unions with $\{p,q\}$ we will call the \textbf{left and right boundaries} of $F$.  
 
 \begin{definition} \label{def:left}
	A marked poset $(P,A,\lambda)$ is called \textbf{strongly planar} if $P$ is strongly planar as an unmarked poset and admits a bounded embedding $H$ such that for each bounded face $F\subseteq \widehat{P}$ of $H$, if the left boundary of $F$ (including $\min(F)$ and $\max(F)$) contains a marked element, then both $\min(F)$ and $\max(F)$ are marked. We will call such an embedding a \textbf{bounded strongly planar embedding.}  \end{definition} 

\begin{remark} We note that in Definition \ref{def:left} we made a choice to put conditions on the markings on the left boundaries of bounded faces. Of course we could have put those conditions on the right boundaries instead. Moreover, as Remark \ref{rem:mix} explains, the definition can be relaxed by mixing and matching left and right boundaries of bounded faces under certain conditions in such a way that the main result, Theorem \ref{thm:markedorderflow}, still holds.
\end{remark}
 
For a bounded strongly planar embedding $H$ of the marked poset $(P,A,\lambda)$, we now construct a flow network $G_{(P,A,\lambda)}$ from $G_P$. Begin with a bounded strongly planar embedding $H$  and the flow network $G_P$ constructed from $H$, as in the case of order polytopes. View the markings $\lambda$ as being on $A$ inside of $H$, and add additional markings $\min\{\lambda(a) \mid  a \in A\}$ on $\hat{0}$ and $\max\{\lambda(a) \mid  a \in A\}$ on $\hat{1}$. 

Recall that each vertex of $G_P$ is naturally labeled by a bounded face of $H$. (In the rest of the paper, whenever we refer to a face of bounded strongly planar embedding $H$, we mean a bounded face.)  Denote the vertex labeled by a face $F$ by $v_F$. Starting from $G_P$, construct a flow network $G_{(P,A,\lambda)}$ by applying the following construction to $v_F$ for each (bounded) face $F$ of $H$. See Figure \ref{fig:gpaconstructionfaceexample} for an illustration of this construction.

If $F$ contains no marked elements on its left boundary, do nothing, and let $v_F$ continue to have netflow 0. Otherwise, suppose the left boundary of $F$ is composed of elements $p_1>\cdots>p_{k}$ in $H$, with $\min(F)=p_{k}$ and $\max(F)=p_1$.  Since some point on the left boundary of $F$ is marked, so are $p_1=q_1$ and $p_k=q_\ell$ by strong planarity. Suppose the marked elements among $p_1,\ldots,p_k$ are $p_1=p_{i_1}>p_{i_2}>\cdots>p_{i_g}=p_k$  marked by $a_1\geq a_2\geq \cdots\geq a_g$. Delete the edges outgoing from  vertex $v_F$ in $G_P$, and let $v_F$ become a sink with netflow $-(p_1-p_k)$, with the incoming edges as before.  The edges previously outgoing from $v_F$ in $G_P$ that crossed the left boundary of $F$ between marked elements $p_{i_m}$ and $p_{i_{m+1}}$ will now be outgoing from the source vertex $s_{m}^F$, for $m \in [g-1]$. Assign $s^F_m$ netflow $a_m-a_{m+1}$ for each $m$.

\begin{figure}[ht]
	\begin{center}
		\includegraphics[scale=1]{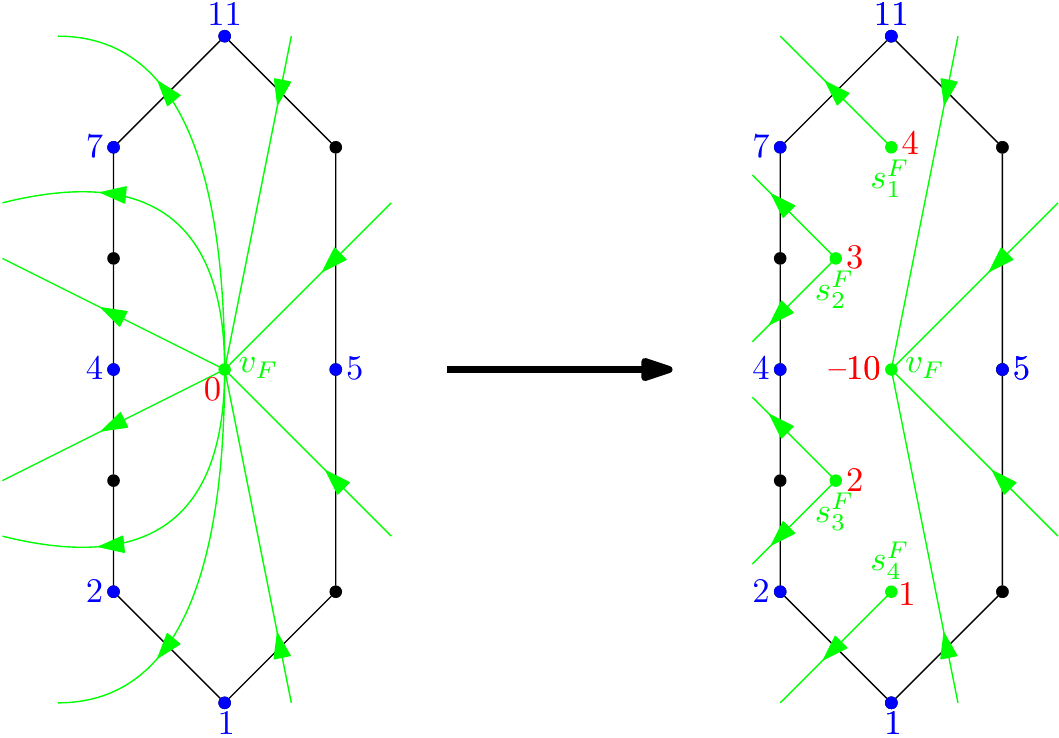}
	\end{center}
	\caption{An illustration of the construction of $G_{(P,A,\lambda)}$ from $G_P$ on a single face $F$.}
	\label{fig:gpaconstructionfaceexample}
\end{figure}

\begin{namedtheorem}[\ref{thm:markedorderflow}] 
	%\label{thm:markedorderflow} 
	Given a bounded strongly planar embedding $H$ of a marked poset $(P,A,\lambda)$, the marked order polytope $\mathcal{O}(P,A)_{\lambda}$ is integrally equivalent to the flow polytope $\F_{G_{(P,A,\lambda)}}$, where $G_{(P,A,\lambda)}$ is the flow network described above.
\end{namedtheorem}
\begin{proof}
	The integral equivalences between $\mathcal{O}(P,A)_\lambda$ and $\mathcal{F}_{G_{(P,A,\lambda)}}$ are exactly as in the order polytope case. The map $\Gamma\colon\mathcal{O}(P,A)_\lambda\to\mathcal{F}_{G_{(P,A,\lambda)}}$ is given by $(x_p)_{p\in P}\mapsto f$ where $f(e)=x_q-x_p$ if $e$ crosses the edge $p<q$ in $H$. The inverse map $\Gamma^{-1}$ is given by $x_p=\sum_{e} f(e)$ over edges $e\in G_{(P,A,\lambda)}$ crossing any fixed path from $\hat{0}$ to $p$ in $H$. (Note that from any marked point $p \in A$, there exists a path from $p$ to $\hat{0}$ in $H$ that only walks along the left boundaries of faces to the minimums of those faces.) The details of the proof are analogous to those in \cite{floworderpolytopes} and are left to the reader. 
\end{proof}

Theorem \ref{thm:markedorderflow} provides a general framework for obtaining the graphs $G_\lambda$ used in Theorem \ref{thm:gtflowpolytope}, and for proving Theorem \ref{thm:gtflowpolytope}. See Figure \ref{fig:gpalambdaexample} for an example of this.
	
\begin{figure}[ht]
	\begin{center}
		\includegraphics[scale=.75]{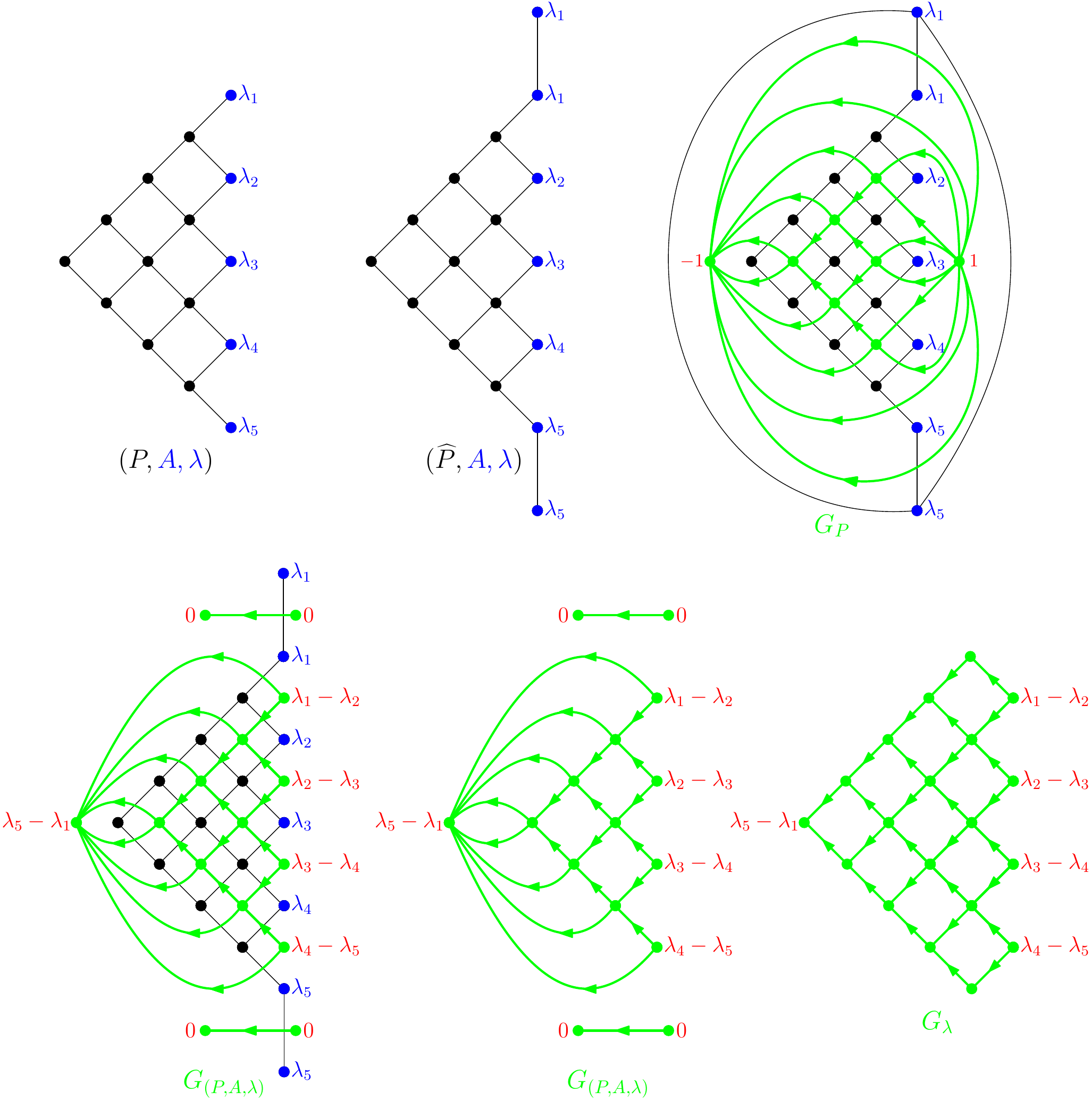}
	\end{center}
	\caption{Top row: the marked poset $(P,A,\lambda)$ with $\mathcal{O}(P,A)_\lambda=\mathrm{GT}(\lambda)$, $\widehat{P}=P\sqcup\{\hat{0},\hat{1}\}$, and the flow network $G_P$;
	Bottom row: the flow network $G_{(P,A,\lambda)}$ and the integrally equivalent flow network $G_\lambda$ of Definition \ref{def:Glambda}.}
	\label{fig:gpalambdaexample}
\end{figure}

 \medskip
 
 \begin{remark} \label{rem:mix}
Note that  Theorem \ref{thm:markedorderflow} can be generalized in various ways. We can obtain slightly different conditions on the markings of strongly planar posets under which the statement of Theorem \ref{thm:markedorderflow} as well as the map given in its proof are still correct. We picked the above particular definition for bounded strongly planar embeddings relying on conditions on the left boundaries of the bounded faces of the embedding as it seemed the least technical to state. We could have, of course, equally worked with right boundaries of the  bounded faces of the embedding, or,  we could mix and match  as to when we consider the left or right boundary of a bounded face as long as we ensure that the flow conditions pick up the restriction coming from two marked points that are comparable but do not lie in a common face.  Next we give an example of how relaxing the marking conditions in    Theorem \ref{thm:markedorderflow}  yields  that skew Gelfand-Tsetlin polytopes are flow polytopes.
\end{remark}

\begin{definition}
	Given partitions $\lambda,\mu\in\mathbb{Z}^n_{\geq 0}$ and $m\in\mathbb{N}$, the \textbf{skew Gelfand-Tsetlin polytope} $\mathrm{GT}(\lambda/\mu,m)$ is the set of all arrays
	\begin{center}
		\includegraphics[scale=1]{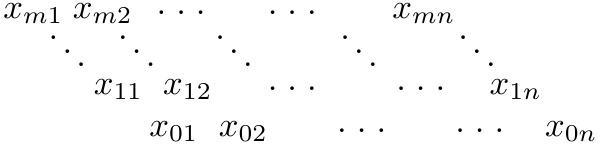}
	\end{center}
with top row $\lambda$ and bottom row $\mu$ such that $x_{ij}\geq x_{i-1,j}$ and $x_{ij}\geq x_{i+1,j+1}$.
\end{definition}
\begin{proposition}
	\label{prop:skewgtposet}
	Skew Gelfand-Tsetlin polytopes are marked order polytopes of strongly planar marked posets.
\end{proposition}
\begin{proof}
	Given $\lambda$, $\mu$, and $m$, begin with a skew Gelfand-Tsetlin array $(x_{ij})_{i,j}$. Replace each entry $x_{ij}$ by a vertex, and each relation $x_{ij}\geq x_{i-1,j}$ or $x_{ij}\geq x_{i+1,j+1}$ by an edge between the corresponding vertices. Mark the top row of vertices with the corresponding entries of $\lambda$, and mark the bottom row of vertices with the corresponding entries of $\mu$. Rotate the graph 90 degrees clockwise. The result is the Hasse diagram of a strongly planar marked poset $(P,A,\lambda)$ with $\mathcal{O}(P,A,\lambda)=\mathrm{GT}(\lambda/\mu,m)$. See Figure \ref{fig:skewgt} for an example of this construction.
\end{proof}
\begin{corollary}
	Skew Gelfand-Tsetlin polytopes are integrally equivalent to flow polytopes.
\end{corollary}
\begin{proof}
	Apply the generalization of Theorem \ref{thm:markedorderflow} explained in Remark \ref{rem:mix} to the poset constructed in Lemma \ref{prop:skewgtposet}. See Figure \ref{fig:skewgt} for an example of the resulting flow network.
\end{proof}

\begin{figure}[ht]
	\includegraphics[scale=.75]{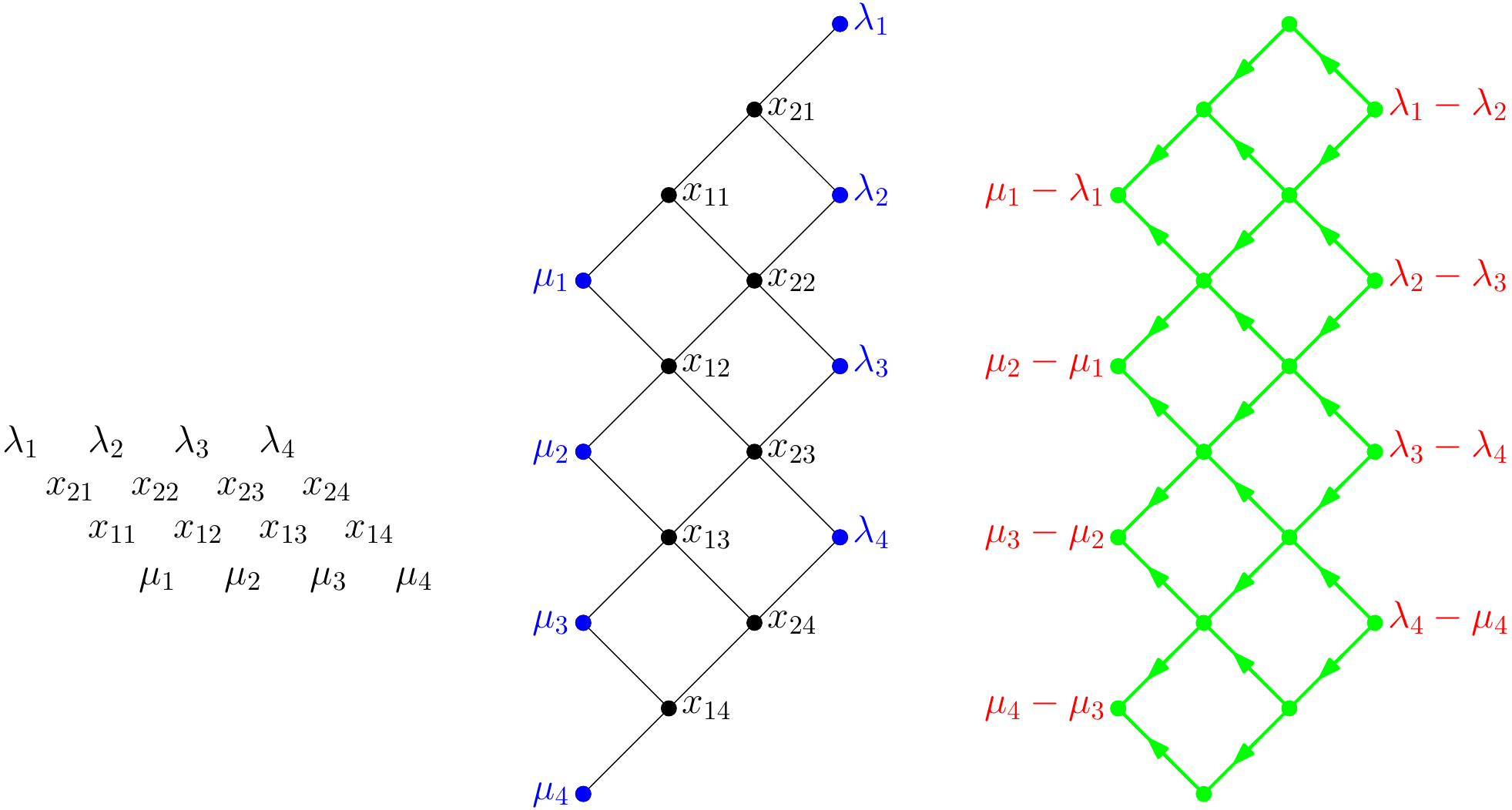}
	\caption{An example of the marked poset and corresponding flow network for recognizing a skew Gelfand-Tsetlin polytope $\mathrm{GT}(\lambda/\mu,3)$ with $\ell(\lambda)=4$ as a marked order polytope and a flow polytope.}
	\label{fig:skewgt}
\end{figure}

\section{Subdivisions of marked order and flow polytopes}
\label{sec:bijection}

In this section, we will give subdivision procedures for $\mathcal{O}(P,A)_\lambda$ and $\mathcal{F}_{G_{(P,A,\lambda)}}$ and prove the two procedures are equivalent. In particular, this will yield a bijective proof of Corollary \ref{cor:bij}.

We start by reviewing the subdivision procedure for flow polytopes following the exposition of  \cite{MMlidskii}. However, we will use a simplified version of the flow polytope subdivision method presented there, specialized to the types of graphs that appear in the present paper. 
 
\subsection{Subdividing  flow polytopes into products of simplices} \label{subsec:flow}
Flow polytopes admit a combinatorial iterative subdivision procedure. To describe the algorithm, we first introduce the necessary terminology and notation. A \textbf{bipartite noncrossing tree} is a tree with a distinguished bipartition of vertices into \textbf{left vertices} $x_1,\ldots,x_{\ell}$ and \textbf{right vertices} $x_{\ell+1},\ldots, x_{\ell+r}$ with no pair of edges $(x_p,x_{\ell+q}), (x_t,x_{\ell+u})$ where $p<t$ and $q>u$. Denote by $\mathcal{T}_{L,R}$ the set of  bipartite noncrossing trees, where $L$ and $R$ are the ordered sets $(x_1,\ldots,x_{\ell})$ and $(x_{\ell+1},\ldots,x_{\ell+r})$ respectively. Note that $|\mathcal{T}_{L,R}|=\binom{\ell+r-2}{\ell-1}$, since they are in bijection with weak compositions of $r-1$ into $\ell$ parts: a tree $T$ in $\mathcal{T}_{L,R}$ corresponds to the composition $(b_1-1,\ldots,b_\ell-1)$ of $r-1$, where $b_i$ denotes the number of edges incident to the left vertex $x_{\ell+i}$ in $T$. 

The bipartite noncrossing tree encoded by the composition $(0,2,1,1)$ is the following:
\begin{center}
	\includegraphics[width=1.3cm]{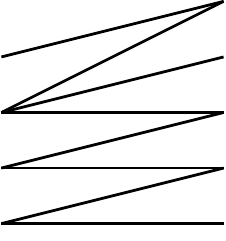}    
\end{center}

Consider a graph $G$ on the vertex set $[n+1]$ and an integer netflow
vector $a=(a_1, \ldots, a_n, -\sum_i a_i)$. In this paper, we will assume that $a_i>0$ implies $i$ has no incoming edges, $a_i<0$ implies $i$ has no outgoing edges, and $a_i=0$ implies $i$ has both incoming and outgoing edges. For these flow networks, the basic step of the subdivision method is the following:

\noindent Pick an arbitrary vertex $i$ of $G$ with netflow $a_i =0$.
Let $\mathcal{I}_i=\mathcal{I}_i(G)$ be the multiset of edges incoming to $i$, edges of the form $(\cdot,i)$. Let $\mathcal{O}_i=\mathcal{O}_i(G)$ be the multiset of outgoing edges from $i$, edges of the form $(i,\cdot)$. 

Assign an ordering to the sets $\mathcal{I}_i$ and $\mathcal{O}_i$, and consider a tree $T \in \mathcal{T}_{\mathcal{I}_i,\mathcal{O}_i}$. For each tree-edge $(e_1,e_2)$ of $T$, where $e_1=(r,i) \in \mathcal{I}_i$ and $e_2=(i,s)\in \mathcal{O}_i$, let $\mathrm{edge}(e_1,e_2)=(r,s)$. Define a graph $G^{(i)}_T$ by starting with $G$, deleting vertex $i$ and all incident edges $\mathcal{I}_i \cup \mathcal{O}_i$ of $G$, and adding the multiset of edges $\{\mathrm{edge}(e_1,e_2) \mid (e_1,e_2)\in E(T)\}$. See Figure \ref{fig:compoundedreduction} for an example. 
 
\begin{lemma}[Compounded Subdivision Lemma] 
	\label{lem:bigreduction} 
	Let  $G$ be a flow network on the vertex set $[n+1]$ with netflow $a=(a_1,\ldots,a_n,-\sum_{i=1}^n a_i)\in \mathbb{Z}^{n+1}$ and a vertex $i\in \{2,\ldots,n\}$ with $a_i=0$. Then, 
	$\{ \F_{G_T^{(i)}}(\hat{a})\}_{{T \in \mathcal{T}_{\mathcal{I}_i,\mathcal{O}_i}}}$ are top dimensional pieces in a subdivision of  
	$\F_G(a)$, where $\hat{a}$ equals $a$ with $ith$ coordinate deleted. 
\end{lemma}

In order to view $\mathcal{F}_{G^{(i)}_T}$ as a subset of $\mathcal{F}_G$, label each edge $e$ of $G$ with a coordinate $f_e$. Label each new edge of $G_T^{(i)}$ by the formal sum of the coordinates of the edges of $G$ that formed it. To get an inclusion $\mathcal{F}_{G_T^{{(i)}}}\subseteq \mathcal{F}_G$, simply add the flow value of each edge in $G_T^{{(i)}}$ to all edges of $G$ appearing in the formal sum labeling it.

We refer to replacing $G$ by $\{G_T^{(i)}\}_{T\in\mathcal{T}_{L,R}}$ as a \textbf{compounded reduction} on $G$. In order to fully subdivide $\mathcal{F}_G$ into simplices, one performs a compounded reduction on $G$, then iteratively performs compound reductions on the graphs $G_T^{(i)}$. A series of these reductions can be efficiently encoded by a \textbf{compounded reduction tree}: the root of the tree is the original graph $G$; and the descendants of any node are the graphs obtained via a compounded reduction on that node. See Figure \ref{fig:compoundedreduction} for an example. The \textbf{canonical compounded reduction tree} of $G$ is obtained by performing compounded reductions from highest to lowest index netflow zero vertices, as in Figure \ref{fig:compoundedreduction}.  There is a natural way of labeling the  products of simplices into which we   subdivide our flow polytope via the compounded reductions by  integer points of other flow polytopes, as is explained in \cite{MMlidskii}.

\begin{figure}[ht]
	\includegraphics[scale=.5]{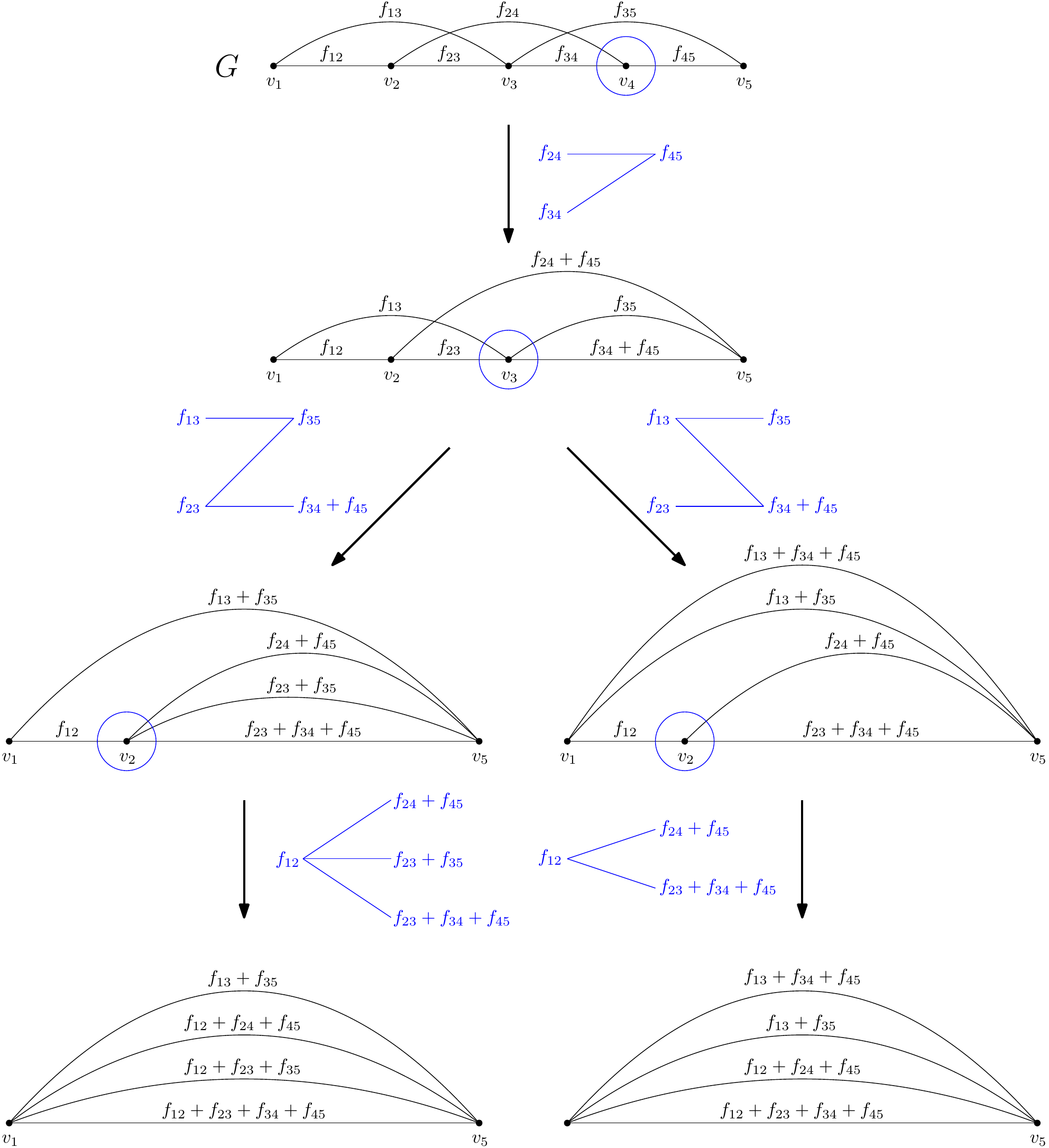}
	\caption{An example of a compounded reduction tree for a graph $G$ with netflow $(1,0,0,0,-1)$. The tree is built downward by performing a series of compounded reductions, starting with the root $G$. The labels on edges record the inclusion maps from the flow polytope of each graph to the flow polytope of $G$.}
	\label{fig:compoundedreduction}
\end{figure}

\subsection{Subdividing  order  polytopes into products of simplices} \label{subsec:order}
Given a bounded strongly planar embedding $H$ of a marked poset $(P, A, \lambda)$, consider the following method for subdividing $\mathcal{O}(P,A)_\lambda$:
\noindent Consider any face $F$ of $H$ not containing an edge $(\hat{0},\hat{1})$ (where, as previously, by face we mean bounded face). Suppose that $F$ is bounded on the left by $p_1>\cdots>p_k$ and on the right by $p_1=q_1>\cdots>q_\ell=p_k$. Replacing $F$ by any of the $\mathcal{N}=\binom{k+\ell-4}{\ell-2}$ linear extensions of $p_1, \ldots, p_{k}, q_2, \ldots, q_{\ell-1}$, we obtain $\mathcal{N}$ strongly planar marked posets $(P_1, A, \lambda), \ldots, (P_{\mathcal{N}}, A, \lambda)$.

\begin{lemma}
	\label{lem:markedorderpolytopesubdivisionstep}
	The marked order polytopes $\mathcal{O}(P_1, A, \lambda), \ldots, \mathcal{O}(P_{\mathcal{N}}, A, \lambda)$ described above form a subdivision of the order polytope $\mathcal{O}(P, A, \lambda)$.
\end{lemma}
\begin{proof}
	This subdivision is obtained by cutting $\mathcal{O}(P,A)_\lambda$ by the hyperplanes $x_{p_i}=x_{q_j}$ for $i \in [1,k]$ and $j\in [1,\ell]$.
\end{proof}

By the above lemma, applying the above construction iteratively to each face  of the bounded strongly planar  embedding $H$ of the marked poset $(P, A, \lambda)$ yields a subdivision of $\mathcal{O}(P,A)_\lambda$ into the marked poset polytopes of a set of marked chains, that is, into products of simplices.

\begin{figure}[ht]
	\includegraphics[scale=.9]{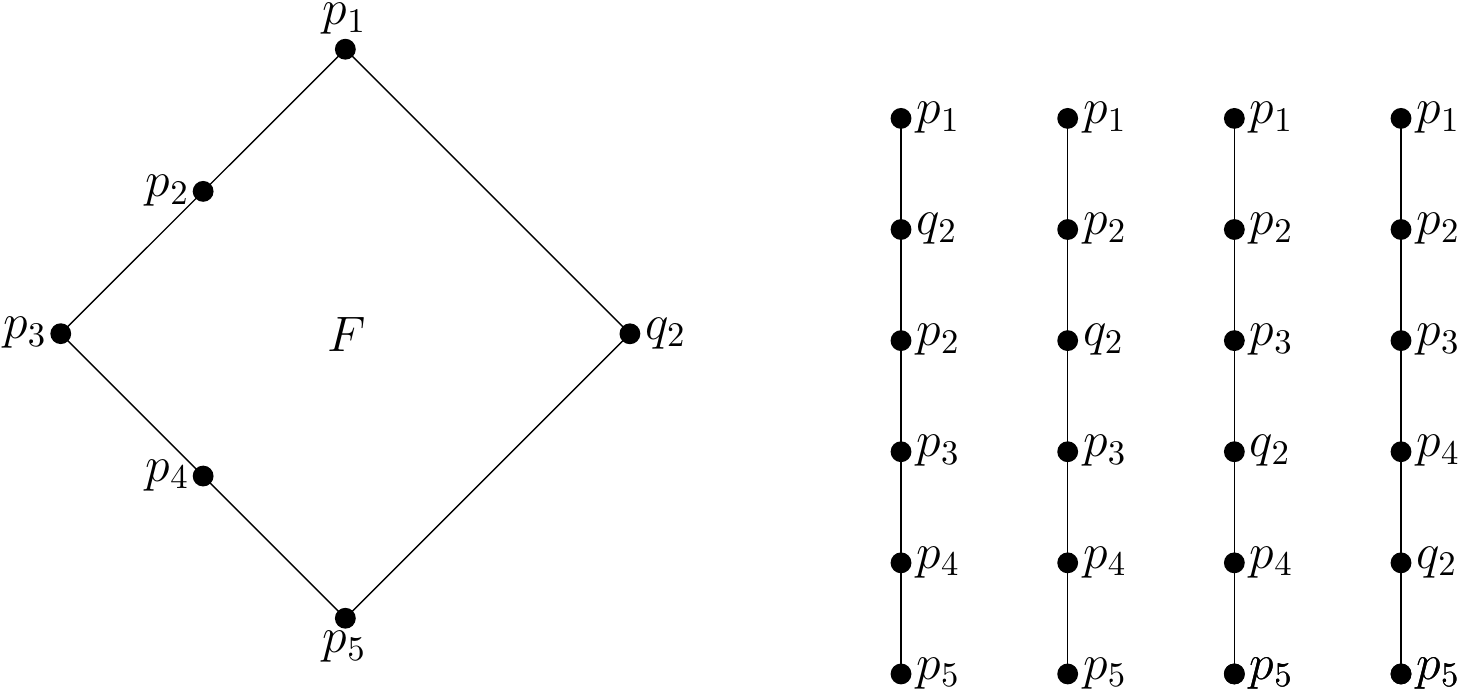}
	\caption{A face $F$ with unmarked left boundary in the embedding of a strongly planar poset $P$ (left) and the corresponding linear extensions of $F$ that replace $F$ to form the posets $P_1,\ldots,P_\mathcal{N}$ when subdividing $\mathcal{O}(P,A,\lambda)$ (right).}
	\label{fig:markedorderpolytopesubdivisionexample}	
\end{figure}

\subsection{Comparing the  subdivisions of flow and order polytopes}
 Theorem \ref{thm:markedorderflow} shows $\mathcal{O}(P,A)_\lambda$ is integrally equivalent to a flow polytope $\mathcal{F}_{G_{(P,A,\lambda)}}$. As we saw in Sections \ref{subsec:flow} and \ref{subsec:order}, both flow and order polytopes admit an iterative subdivision procedure.  We show here that indeed those procedures can be considered identical. 
 
 Through a single application of Lemma \ref{lem:bigreduction} on $\mathcal{F}_{G_{(P,A,\lambda)}}$, we obtain the following.

\begin{lemma}  \label{lem:markedflowpolytopesubdivisionstep} Given a bounded strongly planar embedding $H$ of the marked poset $(P, A, \lambda)$, consider a face $F$ of $H$ which has no markings on its left boundary. Linearly order the $k-1$ outgoing and $\ell-1$ incoming edges of $v_F$ from top to bottom. Performing a compounded reduction at $v_F$ on $G_{(P,A,\lambda)}$ yields $\mathcal{N}=\binom{k+\ell-4}{\ell-2}$ flow networks $G_{(P,A,\lambda)}^1,\ldots,G_{(P,A,\lambda)}^{\mathcal{N}}$ such that the polytopes $\F_{G_{(P,A,\lambda)}^1},\ldots, \F_{G_{(P,A,\lambda)}^{\mathcal{N}}}$ subdivide $\F_{G_{(P,A,\lambda)}}$.
\end{lemma}

\begin{figure}[ht]
	\includegraphics[scale=.9]{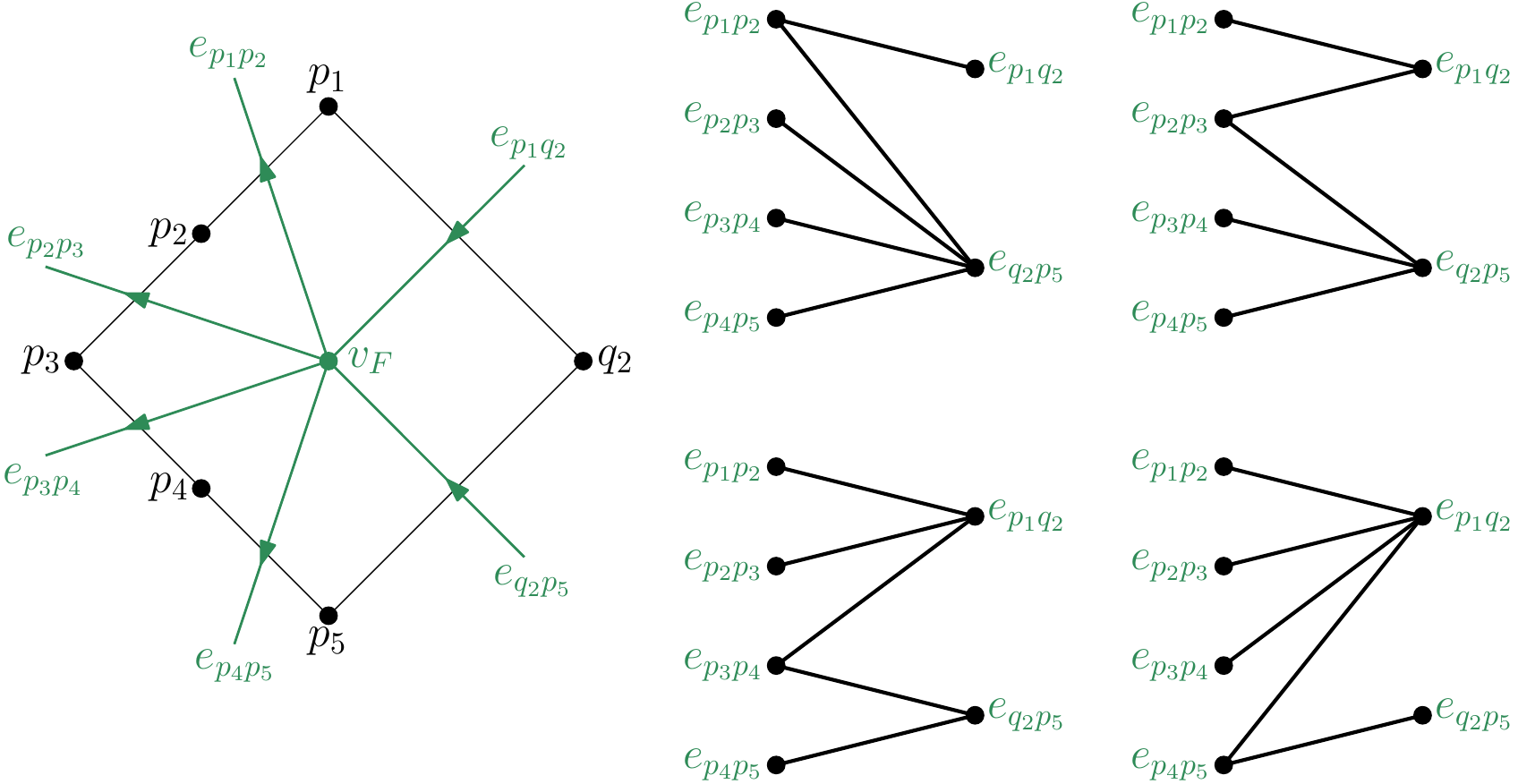}
	\caption{A netflow zero vertex $v_F$ of $G_{(P,A,\lambda)}$ (left) and the corresponding bipartite noncrossing trees that can be used to perform a compounded reduction at $v_F$ to produce flow networks $G^1_{(P,A,\lambda)},\ldots,G^{\mathcal{N}}_{(P,A,\lambda)}$ whose flow polytopes subdivide $\mathcal{F}_{G_{(P,A,\lambda)}}$ (right).}
	\label{fig:markedflowpolytopesubdivisionexample}
\end{figure}

We now describe an equivalence between the subdivision procedures of $\mathcal{O}(P,A)_\lambda$ and $\mathcal{F}_{G_{(P,A,\lambda)}}$ whose basic step is described in Lemma \ref{lem:markedorderpolytopesubdivisionstep} and Lemma \ref{lem:markedflowpolytopesubdivisionstep} respectively. We first focus on the case of a single step of both subdivisions.
 
 \begin{lemma}  
 	\label{lem:bijectioninductivestep} 
 	Given a bounded strongly planar embedding $H$ of a marked poset $(P, A, \lambda)$, let $F$ be a face of $H$ with unmarked left boundary. Let $\Gamma$ be the integral equivalence $\mathcal{O}(P,A)_\lambda\to \mathcal{F}_{G_{(P,A,\lambda)}}$ of Theorem \ref{thm:markedorderflow}. Then there is a bijection $\gamma\colon[\mathcal{N}]\to[\mathcal{N}]$ between the linear extensions of $F$  and the bipartite noncrossing trees from a compounded reduction at $v_F$ such that 
 	\[\Gamma\colon\mathcal{O}(P_j,A)_\lambda \to \mathcal{F}_{G^{\gamma(j)}_{(P,A,\lambda)}}.\]
\end{lemma}
\begin{proof}
	Let $F$ be bounded by $p_1>p_2>\cdots>p_k$ on the left and $p_1=q_1>q_2>\cdots>q_\ell=p_k$ on the right. To define the bijection, we start by drawing the bipartite noncrossing trees with the vertices arranged in vertical columns. Label each vertex of the tree by the edge of $H$ dual to the edge of $G_{(P,A,\lambda)}$ it represents. Encase each tree in a bounding rectangle so that the vertex columns lie on the interiors of the sides of the rectangle. See Figure \ref{fig:subdivisionbijectionexample}. 
	
	To construct a linear order from a tree, we will label the regions of the rectangle cut out by the tree. Label the top region $p_1$ and the bottom $p_k$. All intermediate regions are triangles with exactly one edge on the bounding rectangle. Label such regions by the common label of the endpoints of this edge. The result will be a linear order of the face $F$. Conversely, a linear ordering gives an ordering on the edge segments on each side of the rectangle. Build the tree top to bottom by adding in edges inside the bounding rectangle to cut out regions as specified by the linear order from top to bottom.
	
	To see that $\gamma$ has the property 
	\[\Gamma\colon\mathcal{O}(P_j,A)_\lambda \to \mathcal{F}_{G^{\gamma(j)}_{(P,A,\lambda)}},\]
	it suffices to note that $\gamma$ is constructed precisely so that $G_{(P_j,A,\lambda)}$ = $G^{\gamma(j)}_{(P,A,\lambda)}$ for each $j\in[\mathcal{N}]$.
\end{proof}

\begin{figure}[ht]
	\includegraphics[scale=.90]{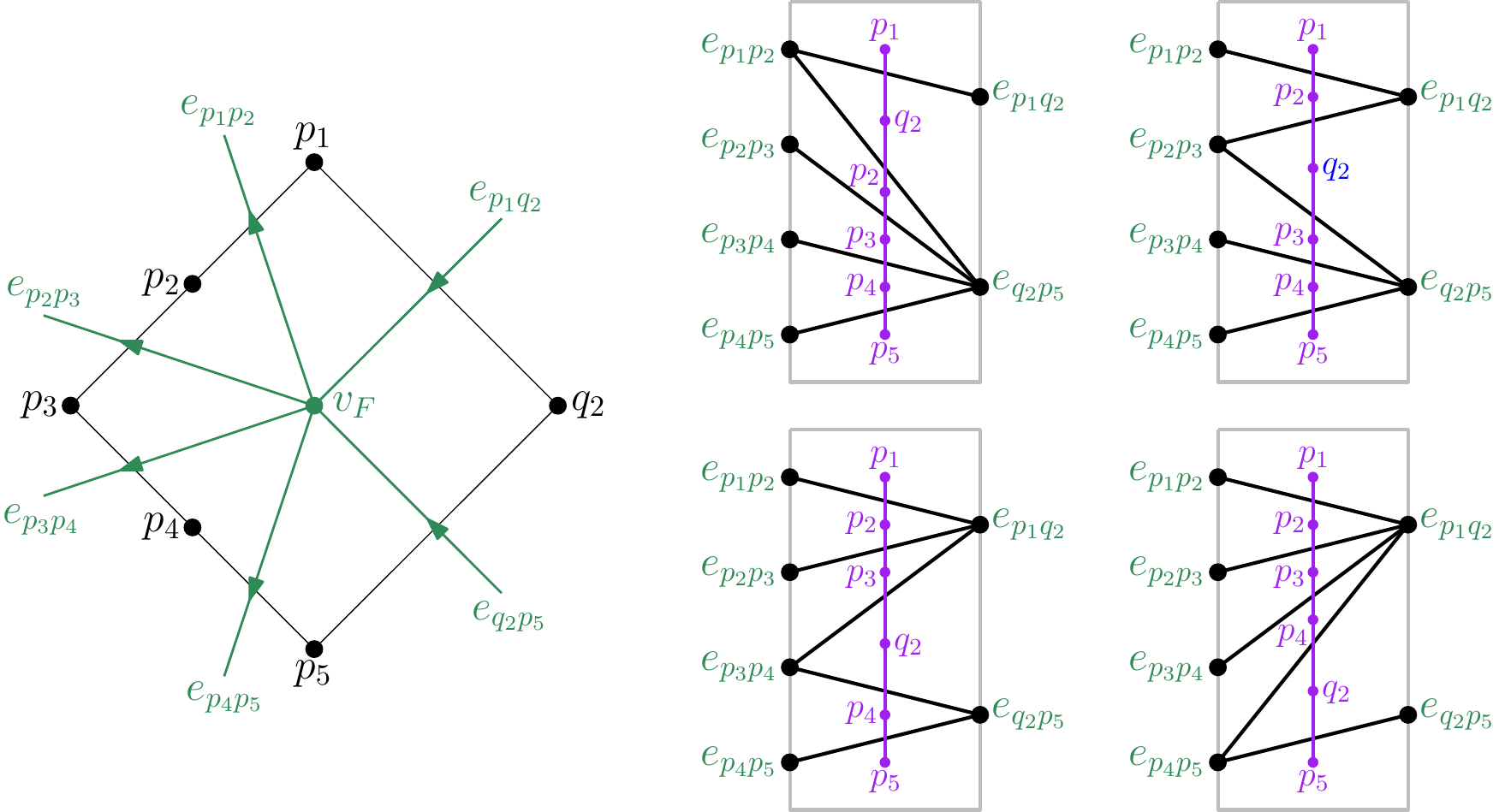}
	\caption{An example of the equivalence between the subdivisions of $\mathcal{O}(P,A)_\lambda$ and $\mathcal{F}_{G_{(P,A,\lambda)}}$ through a bijection of linear orders and bipartite noncrossing trees.}
	\label{fig:subdivisionbijectionexample}
\end{figure}

\begin{theorem}
	Let $H$ be a bounded strongly planar embedding of the marked poset $(P, A, \lambda)$. Choose an ordering $F_1,\ldots, F_m$ of the faces of $H$ that contain no marked elements on their respective left boundaries. Let $\Delta$ be the subdivision of $\mathcal{O}(P,A)_\lambda$ obtained by applying Lemma \ref{lem:markedorderpolytopesubdivisionstep} to each of $F_1,\ldots,F_m$. Let $\Delta'$ be the subdivision of $\mathcal{F}_{G_{(P,A,\lambda)}}$ obtained by applying Lemma \ref{lem:markedflowpolytopesubdivisionstep} to each of $v_{F_1},\ldots,v_{F_m}$ in $G_{(P,A,\lambda)}$. Then the integral equivalence $\Gamma\colon\mathcal{O}(P,A)_\lambda\to \mathcal{F}_{G_{(P,A,\lambda)}}$ induces a bijection $\gamma$ from regions of $\Delta$ to regions of $\Delta'$. 
\end{theorem}
\begin{proof}
	Apply Lemma $\ref{lem:bijectioninductivestep}$ iteratively to each of $P_1,\ldots,P_\mathcal{N}$.
\end{proof}

 \begin{figure}
	\includegraphics[scale=.75]{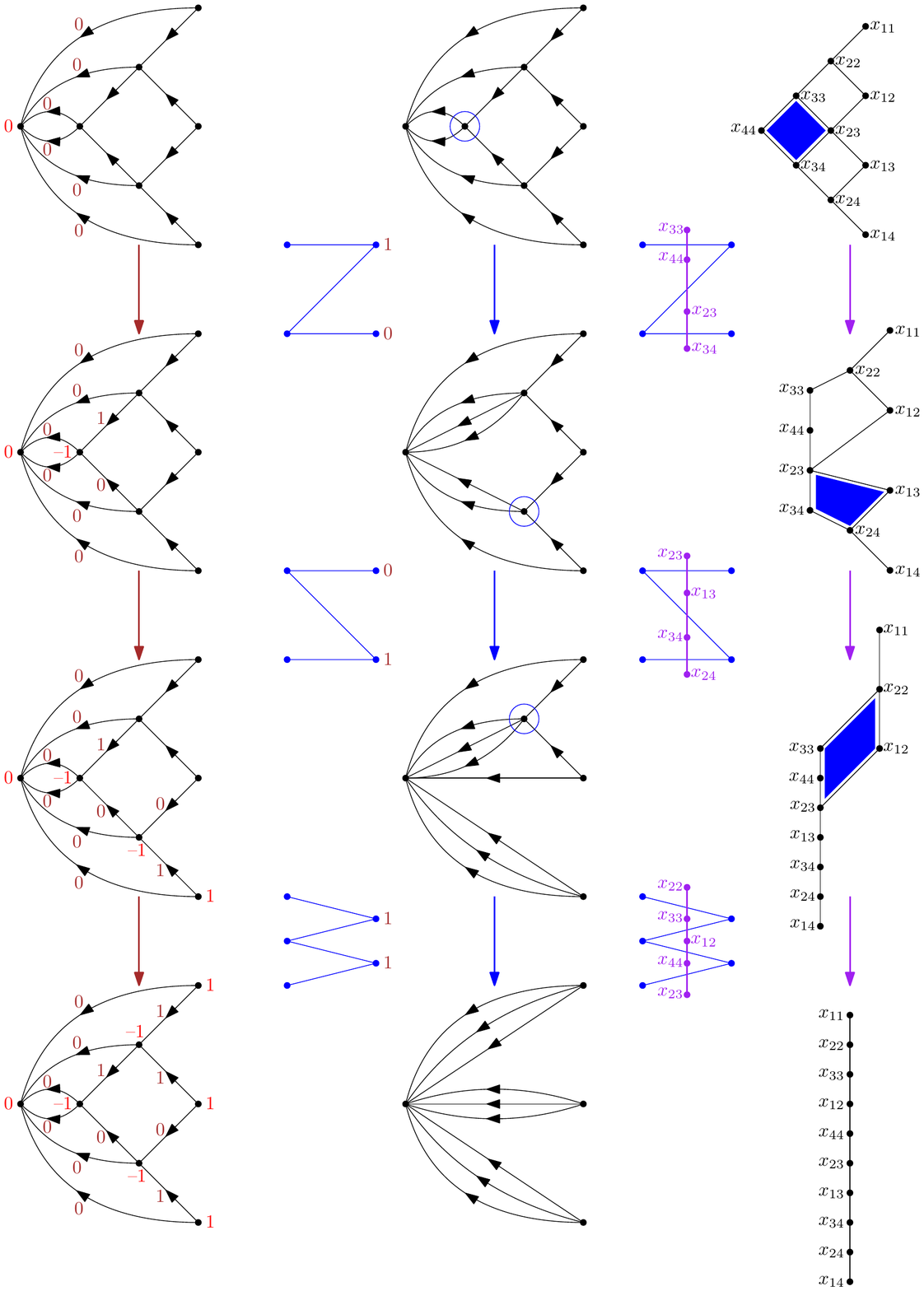}
	\caption{An example of the bijective proof of Corollary \ref{cor:generalbij}. In the left column, a flow on $G_{(P,A,\lambda)}$ is mapped to a leaf of the canonical compounded reduction tree of $G_{(P,A,\lambda)}$. Then, each step of the path (middle column) from the root to that leaf is mapped to a marked order polytope by $\gamma$ (right column), the last of which is a linear order.}
	\label{fig:bijectionexample}
\end{figure}

 \subsection{Bijecting the combinatorial objects labeling the subdivisions of flow and order polytopes}
Now we are ready to give a bijective proof of a generalization of Corollary \ref{cor:bij}. Figure \ref{fig:bijectionexample} provides a detailed example of the bijection. 
	
  \begin{corollary}
 	\label{cor:generalbij}
	Let $H$ be a bounded strongly planar embedding of the marked poset $(P,A,\lambda)$ with $A=\{p_1,\ldots,p_k\}$ such that $\lambda(p_1)\geq \ldots\geq \lambda(p_k)$. Additionally, assume $P$ is marked in such a way that $G_{(P,A,\lambda)}$ has only one sink. Order the $n$ vertices of $G_{P,A,\lambda}$ so that sources corresponding to $p_1,\ldots,p_k$ are first, edges go from earlier to later vertices, and the sink is last. Then 
	\[N_{P,A,\lambda}(a_1,\ldots,a_{k-1}) = K_{G_{(P,A,\lambda)}}(a_1-\om_1,\ldots,a_{k-1}-\om_{k-1},-\om_k,\ldots,-\om_{n-1},0), \]
	where $\om_j$ is the outdegree of vertex $j$ in $G_{(P,A,\lambda)}$ minus $1$.
 \end{corollary}
 \begin{proof}
 	Choose an ordering of the vertices of $G_{(P,A,\lambda)}$ so that all edges go from earlier to later vertices in the order. Let $v_{F_{i_1}}<\cdots<v_{F_{i_m}}$ be the induced order of the vertices $v_F$ corresponding to faces with unmarked left boundary. Let $\Delta$ be the subdivision of $\mathcal{O}(P,A)_\lambda$ and $\Delta'$ the subdivision of $\mathcal{F}_{G_{(P,A,\lambda)}}$ obtained by using Lemma \ref{lem:markedorderpolytopesubdivisionstep} on $F_{i_m},\ldots,F_{i_1}$ and Lemma \ref{lem:markedflowpolytopesubdivisionstep} on $v_{F_{i_m}},\ldots,v_{F_{i_1}}$ respectively. The integral equivalence $\Gamma\colon\mathcal{O}(P,A)_\lambda\to \mathcal{F}_{G_{(P,A,\lambda)}}$ induces a bijection $\gamma$ from regions of $\Delta$ to regions of $\Delta'$.
 	
	As described in \cite{MMlidskii} Lemma 4.1, flows on $G_{P,A,\lambda}$ with netflow \[(a_1-\om_1,\ldots,a_{k-1}-\om_{k-1},-\om_k,\ldots,-\om_{n-1},0)\] 
	are in bijection with leaves of the canonical compounded reduction tree of $G_{(P,A,\lambda)}$ with $a_i$ edges outgoing from the $i$th source vertex. The flow values on edges incoming to each vertex are read off from the composition corresponding to the noncrossing bipartite tree chosen when reducing that vertex. The volume-preserving bijection $\gamma$ provides a correspondence between these leaves and linear extensions of $P$ with the marked elements in positions  $1,2+a_1,\ldots,k+a_1+\cdots+a_{k-1}$.
 \end{proof}
 See Figure \ref{fig:bijectionexample} for an illustration of Corollary \ref{cor:generalbij}.
 
 \newpage
 
\bibliographystyle{plain}
\bibliography{GT-biblio}
\end{document}